\documentclass{amsart}
\usepackage{amssymb}
\usepackage[all]{xy}
\input{tableau}

\renewcommand{\upsilon}{\varphi}

\newcommand{\rc}{\gentabbox{1}{1}{0.4pt}{2pt}{0.4pt}{0.4pt}{{}}}
\newcommand{\lc}{\gentabbox{1}{1}{2pt}{0.4pt}{0.4pt}{0.4pt}{{}}}
\newcommand{\bc}{\boldtabbox{1}{1}{{}}}
\newcommand{\row}[1]{\hbox{$\tabstyle #1$}}
\newcommand{\q}{\genblankbox{1}{1}\relax}

\newtheorem{thm}{Theorem}[section]
\newtheorem{lem}[thm]{Lemma}
\newtheorem{prop}[thm]{Proposition}
\newtheorem{cor}[thm]{Corollary}
\newtheorem{conj}[thm]{Conjecture}

\theoremstyle{definition}
\newtheorem{defn}[thm]{Definition}

\theoremstyle{remark}

\newtheorem{exam}[thm]{Example}

\numberwithin{equation}{section}

\newcommand{\B}{\mathrm{B}}
\newcommand{\GL}{\mathrm{GL}}
\newcommand{\IC}{\mathrm{IC}}
\newcommand{\Aa}{\mathbb{A}}
\newcommand{\Pp}{\mathbb{P}}

\newcommand{\F}{\mathbb{F}}
\newcommand{\N}{\mathbb{N}}
\newcommand{\Oo}{\mathbb{O}}
\newcommand{\Q}{\mathbb{Q}}
\newcommand{\Z}{\mathbb{Z}}
\newcommand{\cB}{\mathcal{B}}
\newcommand{\cC}{\mathcal{C}}

\newcommand{\cF}{\mathcal{F}}
\newcommand{\cG}{\mathcal{G}}
\newcommand{\cH}{\mathcal{H}}

\newcommand{\cN}{\mathcal{N}}
\newcommand{\cO}{\mathcal{O}}
\newcommand{\cP}{\mathcal{P}}
\newcommand{\cQ}{\mathcal{Q}}

\newcommand{\fN}{\mathfrak{N}}

\newcommand{\tK}{\widetilde{K}}
\newcommand{\Qlb}{\overline{\Q}_\ell}

\newcommand{\bt}{{\mathbf{t}}}
\newcommand{\isomto}{\overset{\sim}{\rightarrow}}

\DeclareMathOperator{\End}{End}
\DeclareMathOperator{\Ind}{Ind}
\DeclareMathOperator{\Irr}{Irr}
\DeclareMathOperator{\Span}{span}
\DeclareMathOperator{\Lie}{Lie}
\DeclareMathOperator{\tr}{tr}

\DeclareMathOperator{\codim}{codim}

\title[Orbit closures in the enhanced nilpotent cone]
{Orbit closures in the enhanced nilpotent cone}
\author{Pramod N. Achar}
\address{Department of Mathematics\\
  Louisiana State University\\
  Baton Rouge, LA 70803\\
  U.S.A.}
\email{pramod@math.lsu.edu}
\author{Anthony Henderson}
\address{School of Mathematics and Statistics\\
  University of Sydney, NSW 2006\\
  Australia}
\email{anthonyh@maths.usyd.edu.au}

\begin{document}

\begin{abstract}
We study the orbits of $G=\GL(V)$ in the enhanced nilpotent cone
$V\times\cN$, where $\cN$ is the variety of nilpotent endomorphisms of $V$.
These orbits are parametrized by bipartitions of $n=\dim V$, and we prove that the
closure ordering corresponds to a natural partial order on bipartitions. 
Moreover, we prove that the
local intersection cohomology of the orbit closures is given by certain
bipartition analogues of Kostka polynomials, defined by Shoji. Finally,
we make a connection with Kato's exotic nilpotent cone in type C,
proving that the closure ordering is the same, and conjecturing that
the intersection cohomology is the same but with degrees doubled.
\end{abstract}

\maketitle

\section{Introduction}
\label{sect:intro}

Many features of the representation theory of an algebraic group are known
to be controlled by the geometry of its nilpotent cone. In particular,
the Springer correspondence, as developed by Borho--MacPherson and Lusztig,
relates the local intersection cohomology of the nilpotent orbit closures
to composition multiplicities in representations of the associated Weyl group
(see the survey article \cite{shoji:survey}). The correspondence in 
types B/C is more complicated than that in type A, in a number of
respects: for instance, Weyl group representations are no longer in
bijection with nilpotent orbits, and the concise algebraic description of
the Weyl group action on the total cohomology of Springer fibres
(see~\cite{dp}, \cite{garsiaprocesi}) is lost.

Recently, S.~Kato has constructed an ``exotic Springer correspondence'' in
type C
(see \cite{kato:springer}) which appears to evade these complications. 
He replaces the
nilpotent cone with the exotic nilpotent cone $\fN=W\times\fN_0$, 
where $W$ is the symplectic vector space
and $\fN_0$ is the variety of self-adjoint nilpotent endomorphisms of $W$.
He had introduced this exotic nilpotent cone
in \cite{kato:exotic}, to generalize the Kazhdan--Lusztig--Ginzburg
geometrization of affine Hecke algebras to the case of unequal parameters.
Thus, questions in the representation theory of the Coxeter group of type B/C,
and of the affine Hecke algebras of type B/C with unequal parameters, are related to
the problem of computing local intersection cohomology of orbit closures in $\fN$.
(Part of this problem, the computation of IC of orbit closures in $\fN_0$,
was done in \cite{hend:ft}.)

As a step towards solving this problem, we consider an analogous
but simpler variety, the \emph{enhanced nilpotent cone} $V\times\cN$, where
$\cN$ is the ordinary nilpotent cone, \textit{i.e.}, the
variety of nilpotent endomorphisms of the vector space $V$.
This enhanced nilpotent cone can be viewed as a subvariety of the exotic nilpotent
cone, of the kind which plays an important role in \cite{kato:exotic}; it is
also important in the theory of
mirabolic character sheaves being developed by Finkelberg, Ginzburg and Travkin
(see \cite{fg}, \cite{travkin}, \cite{fgt}).
The group acting on it
is merely $G=\GL(V)$, so the geometry has the flavour of type A,
whereas the combinatorics arising is of type B/C, in accordance with Kato's picture.
The great advantage of the enhanced nilpotent cone over the exotic is that there is
a standard way to construct resolutions of singularities of the orbit closures,
which turn out to be semismall; these allow us to determine the closure
ordering and the local intersection cohomology.

On the ordinary (type A) nilpotent cone, the combinatorics of the closure order is of course well known, as is Lusztig's identification of the local intersection cohomology with Kostka polynomials (see \cite{lusztig:greenpolys}).  
These results motivate and guide the developments of the present paper.  The ``mirabole'' of this story is that the $G$-action
on pairs of a vector and a nilpotent endomorphism is so similar to that 
on nilpotent endomorphisms \textit{tout court}. 

Here are the principal results of the paper in detail.
\begin{itemize}
\item[\S2.] \textbf{Parametrization of Orbits.} We show that $G$-orbits in
$V\times\cN$ are parametrized by the bipartitions $(\mu;\nu)$ of $n$, where
$n=\dim V$ (Proposition \ref{prop:param}). 
We also show that the orbit $\cO_{\mu;\nu}$ has dimension
$n^2-b(\mu;\nu)$, and that its point stabilizers are connected
(Proposition \ref{prop:dim}). 
The same parametrization of orbits was independently obtained in~\cite{travkin}.  The finiteness of the number of orbits has
been known since~\cite{bern} (see also~\cite[2.1]{ganginzburg}), and
\cite{kato:springer} proved analogous
results for the exotic nilpotent cone.
\item[\S3.] \textbf{Orbit Closures.} We construct resolutions of singularities
of the orbit closures $\overline{\cO_{\mu;\nu}}$ (Proposition \ref{prop:resolution}),
and use them to show (Theorem \ref{thm:closure})
that the closure ordering corresponds to a natural partial order on bipartitions
(Definition \ref{defn:ineq}), which appeared previously in \cite{shoji:limit}.
\item[\S4.] \textbf{Fibres of the Resolutions of Singularities.}
We show that the resolutions of singularities are semismall, and deduce that
the local intersection cohomology can be determined from the cohomology
of the fibres of the resolutions (Theorem \ref{thm:semismall}).
We then show that these fibres can be paved by affine spaces 
(Theorem \ref{thm:paving}), which implies the vanishing of odd-degree cohomology.
\item[\S5.] \textbf{Intersection Cohomology and Kostka Polynomials.}
In the main result of the paper (Theorem \ref{thm:main}), we prove that
for $(v,x)\in\cO_{\rho;\sigma}$,
\begin{equation*} \label{eqn:intro}
t^{b(\mu;\nu)}\sum_i \dim \cH^{2i}_{(v,x)}\IC(\overline{\cO_{\mu;\nu}})\,t^{2i}
=\tK_{(\mu;\nu),(\rho;\sigma)}(t),
\end{equation*}
where the right-hand side is a type-B Kostka polynomial which was
defined by Shoji in \cite{shoji:green}.
\item[\S6.] \textbf{Connections with Kato's Exotic Nilpotent Cone.}
After recalling Kato's result that the orbits in the exotic nilpotent
cone are parametrized by bipartitions, we prove that the closure ordering
in the exotic nilpotent cone is the same as for the enhanced nilpotent cone
(Theorem \ref{thm:moby-dick}), and we conjecture (Conjecture \ref{conj:doubling})
that the local IC is also the same but with degrees doubled (this is the
relationship which is known to hold between $\fN_0$ and $\cN$). We also explain
that this conjecture may be equivalent to one made by Shoji in \cite{shoji:limit}.
\end{itemize} 

\textbf{Acknowledgements.} We are very grateful to S.~Kato, and to
M.~Finkelberg, V.~Ginzburg, and R.~Travkin, for generously keeping us informed of their
work as ours progressed. We are also indebted to T.~Shoji for pointing out
the connection with his paper \cite{shoji:limit}.

\section{Parametrization of Orbits}
\label{sect:param}

The following notation will be in force throughout the paper:
\[
\begin{split}
\F&\text{ is an algebraically closed field,}\\
V&\text{ is an $n$-dimensional vector space over }\F,\\
G&=\GL(V),\text{ and}\\
\cN&=\{x\in\End(V)\,|\,x\text{ is nilpotent}\}.
\end{split}
\]
Given $x\in\cN$, we regard $V$ as an $\F[x]$-module in the obvious way,
where $\F[x]$ is the subalgebra of $\End(V)$ generated by $x$. 
All complexes of sheaves will be $G$-equivariant constructible complexes of $\Qlb$-sheaves, where $\ell$ is a fixed prime not equal to the
characteristic of $\F$.

Our conventions for partition combinatorics follow \cite{macdonald}
in most respects.
A \emph{partition} is a nonincreasing
sequence of nonnegative integers $\lambda = (\lambda_1,\lambda_2,
\cdots)$ with finitely many nonzero terms.  Its \emph{size}, denoted
$|\lambda|$, is the sum of its terms: $\sum_i \lambda_i$.  Its \emph{length},
denoted $\ell(\lambda)$, is the number of nonzero terms. The \emph{transpose
partition} $\lambda^\bt$ is defined by $\lambda_i^\bt=|\{j\,|\,\lambda_j\geq i\}|$.
The set of all
partitions of size $n$ is denoted $\cP_n$; this is a poset under the
dominance order $\leq$, defined so that $\lambda\leq\mu$ is equivalent to
\[
\lambda_1+\lambda_2+\cdots+\lambda_k\leq\mu_1+\mu_2+\cdots+\mu_k,
\text{ for all $k\geq 1$.}
\]
(Note that we never relate partitions of different size in this way, so
$\lambda\leq\mu$ entails $|\lambda|=|\mu|$.)
Addition of partitions is defined termwise: to say that $\lambda = \mu +
\nu$ is to say that $\lambda_i = \mu_i + \nu_i$ for each $i$. Finally,
given a partition $\lambda$, we define
\[
n(\lambda) = \sum_i (i-1) \lambda_i = \sum_i \binom{\lambda_i^\bt}{2}.
\]
This function is clearly additive: $n(\mu + \nu) = n(\mu) + n(\nu)$.  

It is well known that $G$-orbits in $\cN$ are in bijection with $\cP_n$,
via the Jordan normal form.  Explicitly, the orbit $\cO_\lambda$
corresponding to $\lambda \in \cP_n$ consists of all $x \in \cN$ for which
there exists a basis
\[
\{v_{ij} \mid \text{$1 \le i \le \ell(\lambda)$ and $1 \le j \le
\lambda_i$} \}
\quad
\text{for $V$ such that}
\quad
xv_{ij} =
\begin{cases}
v_{i,j-1} & \text{if $j > 1$,} \\
0 & \text{if $j = 1$.}
\end{cases}
\]
We will refer to such a basis $\{v_{ij}\}$ as a \emph{Jordan basis} for $x$,
and to $\lambda$ as the \emph{Jordan type} of $x$; this terminology applies
when $x$ is a nilpotent endomorphism of any finite-dimensional vector space, 
not necessarily our chosen vector space $V$.

To classify the $G$-orbits in $V\times\cN$ we introduce some analogous definitions.

\begin{defn}
A \emph{bipartition} is an ordered pair of partitions, written $(\mu;\nu)$.
The set of bipartitions $(\mu;\nu)$ with $|\mu| + |\nu| = n$ is denoted $\cQ_n$.
Following~\cite[\S 5.5.3]{gp}, for any $(\mu;\nu)\in\cQ_n$, we define
\[
b(\mu;\nu)=2n(\mu)+2n(\nu)+|\nu|.
\]
\end{defn}

\begin{defn}
Let $(v,x)\in V\times\cN$, and let $\lambda$ be the Jordan type of $x$.
A \emph{normal basis} for $(v,x)$ is a Jordan basis $\{v_{ij}\}$ for $x$
such that
\[
v = \sum_{i = 1}^{\ell(\mu)} v_{i,\mu_i},
\]
where $\mu$ is a partition such that $\nu_i=\lambda_i-\mu_i$ defines
a partition $\nu$. The bipartition
$(\mu;\nu)\in\cQ_n$ is the \emph{type} of the normal basis.
\end{defn}

The following result (which holds over non-algebraically closed fields
as well) was independently proved by Travkin (\cite[Theorem 1]{travkin}). 

\begin{prop}\label{prop:param}
The set of $G$-orbits in $V \times \cN$ is in one-to-one correspondence
with $\cQ_n$.  The orbit corresponding to $(\mu;\nu)$, denoted
$\cO_{\mu;\nu}$, consists of pairs $(v,x)$ for which there exists a normal basis
of type $(\mu;\nu)$. 
\end{prop}
\begin{proof}
This proposition follows from the next two lemmas.  Lemma~\ref{lem:jordan}
states that for any pair $(v,x)$, a normal basis exists. It is obvious that
for any $(\mu;\nu)\in\cQ_n$, there exists a pair possessing a normal basis
of that type, and any two such pairs are in the same
$G$-orbit. To complete the proof, we must show that the type of the normal basis
is determined uniquely by $(v,x)$. But the partition $\lambda=\mu+\nu$
is determined as the Jordan type of $x$, and Lemma~\ref{lem:jordan-dist} shows 
that the partition $(\nu_1+\mu_2,\nu_2+\mu_3,\cdots)$ of size $n-\mu_1$
is also determined. Knowing these two partitions, one can successively 
determine $\mu_1,\nu_1,\mu_2,\nu_2$, and so forth.
\end{proof}

\begin{lem}\label{lem:jordan}
For any $(v,x)\in V\times\cN$, there exists a normal basis for
$(v,x)$ of some type $(\mu;\nu)\in\cQ_n$.
\end{lem}
\begin{proof}
Let $\lambda$ be the Jordan type of $x$, and let $\{v_{ij}\}$ be a Jordan basis
for $x$. Write $v=\sum_{i,j} c_{ij} v_{ij}$. For $1\leq i\leq \ell(\lambda)$,
let $\mu_i\in\{0,1,\cdots,\lambda_i\}$ be minimal such that 
$c_{ij}=0$ whenever $\mu_i<j\leq\lambda_i$, and set $\nu_i=\lambda_i-\mu_i$. 
If $\mu_i\neq 0$, we change the basis of the $i$th Jordan block
as follows. Define
\[
v_{i,\lambda_i}'=\sum_{j=1}^{\mu_i} c_{ij} v_{i,j+\nu_i}
\quad
\text{and}
\quad
v_{ij}'=x^{\lambda_i-j}v_{i,\lambda_i}'
\text{ for $1\leq j\leq \lambda_i-1$,}
\]
and then redefine $v_{ij}$ to be $v_{ij}'$. This gives a new
Jordan basis for $x$ with the property that
\[
v=\sum_{\substack{1\leq i\leq\ell(\lambda)\\\mu_i\neq 0}} v_{i,\mu_i}.
\]
If $(\mu_1,\mu_2,\cdots)$ and $(\nu_1,\nu_2,\cdots)$ are partitions, we are
finished. If they are not, we must adjust our basis in an appropriate way.
Arguing by induction on $\ell(\lambda)$, we can assume that $\mu_2\geq\mu_3\geq\cdots$
and $\nu_2\geq\nu_3\geq\cdots$ hold, so the only possible problems are
that $\mu_1<\mu_2$ or that $\nu_1<\nu_2$. Since $\lambda_1\geq\lambda_2$,
these cases are mutually exclusive.

If $\mu_1<\mu_2$, we move the second Jordan block by redefining
$v_{21},v_{22},\cdots,v_{2,\lambda_2}$
to be
\[
v_{21}-v_{11},v_{22}-v_{12},\cdots,
v_{2,\lambda_2}-v_{1,\lambda_2}.
\]
After this change, we still have a Jordan basis for $x$, but the
component of $v$ in the first Jordan block is now $v_{1,\mu_1}+v_{1,\mu_2}$
(or $v_{1,\mu_2}$, if $\mu_1=0$).
Changing the basis of the first Jordan block as above, we can
make this component $v_{1,\mu_2}$. So we have effectively redefined
$\mu_1$ to equal $\mu_2$ and $\nu_1$ to equal $\lambda_1-\mu_2$, 
and thus removed the problem (without making
$\nu_1<\nu_2$).

If $\nu_1<\nu_2$, we move the first Jordan block by redefining
$v_{11},\cdots,v_{1,\lambda_1}$ to be
\[
\begin{split}
v_{11},\cdots,v_{1,\lambda_1-\lambda_2},
v_{1,\lambda_1-\lambda_2+1}-v_{21},
&v_{1,\lambda_1-\lambda_2+2}-v_{22},
\cdots,
\\
&
v_{1,\mu_1}-v_{2,\lambda_2-\nu_1},
\cdots,
v_{1,\lambda_1}-v_{2,\lambda_2}.
\end{split}
\]
After this change, the component of $v$ in the second Jordan block
is $v_{2,\mu_2}+v_{2,\lambda_2-\nu_1}$ (or $v_{2,\lambda_2-\nu_1}$, if $\mu_2=0$).
Changing the basis of the second Jordan block as above, we can make this
component $v_{2,\lambda_2-\nu_1}$. So we have effectively redefined
$\mu_2$ to equal $\lambda_2-\nu_1$ and $\nu_2$ to equal $\nu_1$. The inequalities 
$\mu_1\geq\mu_2\geq\cdots$ remain true, but it
is possible that we now have $\nu_2<\nu_3$; if so, we repeat this procedure
with the second and third Jordan blocks, and continue until we arrive at
the desired result.
\end{proof}

\begin{lem}\label{lem:jordan-dist}
Suppose $\{v_{ij}\}$ is a normal basis for $(v,x)\in V\times\cN$ of type
$(\mu;\nu)\in\cQ_n$. For $1\leq i\leq\ell(\mu+\nu)$ and $1\leq j\leq (\mu+\nu)_i$,
define $w_{ij}=\sum_{k=1}^i v_{k,j-\mu_i+\mu_k}$. 
\begin{enumerate}
\item $\{w_{ij}\,|\,1\leq j\leq \mu_i-\mu_{i+1}\}$ 
is a basis for the $\F[x]$-submodule $\F[x]v$ of $V$,
and therefore $\dim \F[x]v=\mu_1$.
\item $\{w_{i,j+\mu_i-\mu_{i+1}}+\F[x]v\,|\,1\leq j\leq\nu_i+\mu_{i+1}\}$ 
is a Jordan basis for the
induced endomorphism $x|_{V/\F[x]v}$, whose Jordan type is
$(\nu_1+\mu_2,\nu_2+\mu_3,\cdots)$. 
\end{enumerate}
\end{lem}
\begin{proof}
If the basis elements $\{v_{ij}\}$ are drawn in 
a shifted array, where the $i$th row is shifted to the left by $\mu_i$ places,
then $w_{ij}$ is the sum of $v_{ij}$ and all basis elements directly above it.
Hence $\{w_{ij}\,|\,1\leq i\leq\ell(\mu+\nu),1\leq j\leq (\mu+\nu)_i\}$
is another basis of $V$.
For example:
\[
\begin{split}
\mu&=(32^2)\\
\nu&=(21^4)
\end{split}
\qquad
\begin{array}{ccccc}
v_{11}&v_{12}&v_{13}&v_{14}&v_{15}\\
&v_{21}&v_{22}&v_{23}&\\
&v_{31}&v_{32}&v_{33}&\\
&&&v_{41}&\\
&&&v_{51}&\\
\uparrow&\uparrow&\uparrow&&\\
x^2 v&xv&v&&
\end{array}
\]
By definition, we have $v=w_{\ell(\mu),\mu_{\ell(\mu)}}$, the sum of the $0$th
column of the shifted array. Moreover, $x$ takes each basis vector $v_{ij}$ to the
one to the left of it, or to $0$ if $j=1$. Hence 
$x^sv=w_{\mu_{s+1}^\bt,\mu_{\mu_{s+1}^\bt}-s}$ is the sum of the $(-s)$th column 
of the shifted array, for $1\leq s<\mu_1$, and $x^{\mu_1}v=0$. Part (1) follows.
It is also easy to see that
\[ xw_{ij}=\begin{cases}
w_{i,j-1} & \text{if $j \geq \mu_i-\mu_{i+1}+2$,} \\
x^{\mu_i-j+1}v & \text{if $j\leq \mu_i-\mu_{i+1}+1$,}
\end{cases} \]
from which part (2) follows.
\end{proof}

We note some easy facts about this parametrization of $G$-orbits in $V\times\cN$.
\begin{lem} \label{lem:easy}
For any $\lambda\in\cP_n$, the following hold.
\begin{enumerate}
\item The union of the orbits $\cO_{\mu;\nu}$ where $\mu+\nu=\lambda$ is
precisely $V\times\cO_\lambda$.
\item The orbit $\cO_{\varnothing;\lambda}$ is precisely $\{0\}\times\cO_\lambda$.
\item The orbit $\cO_{\lambda;\varnothing}$ consists of all
$(v,x)$ where $x\in\cO_\lambda$ and $v\in V\setminus\ker(x^{\lambda_1-1})$.
\end{enumerate}
\end{lem}

\begin{proof}
Parts (1) and (2) are obvious. Part (3) follows from the observation
that if $(v,x)\in\cO_{\mu;\nu}$ with $\mu+\nu=\lambda$ and $\nu\neq\varnothing$,
then $\mu_1\leq\lambda_1-1$, so $x^{\lambda_1-1}v=0$ by part (1) of Lemma
\ref{lem:jordan-dist}.
\end{proof}

We also need to describe the stabilizers of our group action.
\begin{defn}
For $x\in\cN$, define 
\[
E^x=\{y\in\End(V)\,|\,xy=yx\}\text{ and }
G^x=G\cap E^x.
\]
For $(v,x)\in V\times\cN$, define
\[
E^{(v,x)}=\{y\in E^x\,|\,yv=0\}
\text{ and }G^{(v,x)}=\{g\in G^x\,|\,gv=v\}.
\]
\end{defn}

The first four parts of the next result are well known, but we include them for 
ease of comparison.
\begin{prop}\label{prop:dim}
Let $(\mu;\nu)\in\cQ_n$, and let $\lambda=\mu+\nu$.
Let $(v,x)\in\cO_{\mu;\nu}$, and let $\{v_{ij}\}$ be a normal basis for $(v,x)$.
\begin{enumerate}
\item $E^x$ has basis 
\[
\{y_{i_1,i_2,s}\,|\,1\leq i_1,i_2\leq\ell(\lambda),\,
\max\{0,\lambda_{i_1}-\lambda_{i_2}\}\leq s\leq \lambda_{i_1}-1\},
\]
where
\[
y_{i_1,i_2,s}v_{ij}=\begin{cases}
v_{i_2,j-s},&\text{ if $i=i_1$, $s+1\leq j\leq\lambda_i$,}\\
0,&\text{ otherwise.}
\end{cases}
\]
\item $\dim E^x=n+2n(\lambda)$.
\item $G^x$ is a connected algebraic group of dimension $n+2n(\lambda)$.
\item $\dim \cO_\lambda=n^2-n-2n(\lambda)$.
\item $E^x v=\Span\{v_{ij}\,|\,1\leq i\leq\ell(\mu),1\leq j\leq\mu_i\}$.
\item $\dim E^{(v,x)}=b(\mu;\nu)$.
\item $G^{(v,x)}$ is a connected algebraic group of dimension $b(\mu;\nu)$.
\item $\dim \cO_{\mu;\nu}=\dim \cO_\lambda + |\mu|= n^2 - b(\mu;\nu)$. 
\end{enumerate}
\end{prop}

\begin{proof}
Part (1) is straightforward, and (2) follows easily from (1).
The well-known proof of part (3) is that $G^x$ is the principal open subvariety
of $E^x$ defined by the polynomial function $\det$, which clearly does not vanish
identically. Part (4) follows because $\dim \cO_\lambda=\dim G-\dim G^x$.
To prove part (5), we note that for $i_1,i_2,s$
as in part (1),
\[
y_{i_1,i_2,s}v=
\begin{cases}
v_{i_2,\mu_{i_1}-s},&\text{ if $s+1\leq \mu_{i_1}$,}\\
0,&\text{ otherwise.}
\end{cases}
\]
In the first case, $v_{i_2,\mu_{i_1}-s}$ is in the required subspace because
$\max\{0,\lambda_{i_1}-\lambda_{i_2}\}\geq\mu_{i_1}-\mu_{i_2}$; moreover, every
basis element $v_{ij}$ with $1\leq j\leq\mu_i$ occurs in this way for
$i_1=i_2=i$ and $s=\mu_i-j$, so we have the desired equality. Consequently,
$\dim E^x v=|\mu|$, which implies
\[
\dim E^{(v,x)}=\dim E^x-\dim E^x v=n+2n(\lambda)-|\mu|=b(\mu;\nu),
\]
as required for part (6). Part (7) follows because $G^{(v,x)}$ is the
principal open subvariety of $1_V+E^{(v,x)}$ defined by $\det$.
Finally, we have $\dim \cO_{\mu;\nu}=\dim G-\dim G^{(v,x)}=n^2-b(\mu;\nu)$,
which is part (8).
\end{proof}

We can now state an elegant
alternative characterization of $\cO_{\mu;\nu}$, which is prominent in 
the treatment of Travkin (\cite{travkin}). It is evident \textit{a priori}
that $E^x v$ is an $x$-stable subspace of $V$, so there are induced endomorphisms
$x|_{E^x v}$ and $x|_{V/E^x v}$.

\begin{cor} \label{cor:extension}
Let $(v,x)\in V\times\cN$ and $(\mu;\nu)\in\cQ_n$. Then
$(v,x)\in\cO_{\mu;\nu}$ if and only if the Jordan type of
$x|_{E^x v}$ is $\mu$ and the Jordan type of $x|_{V/E^x v}$ is $\nu$.
\end{cor}

\begin{proof}
Proposition \ref{prop:dim}(5) shows the ``only if'' direction; but the
pair $(\mu;\nu)$ for which $(v,x)\in\cO_{\mu;\nu}$ is uniquely determined
by $(v,x)$, so the ``if'' direction also holds.
\end{proof}

A convenient way to represent a bipartition $(\mu;\nu)$ is as the
`back-to-back union' of the Young diagrams
of $\mu$ and of $\nu$, with a solid vertical line dividing the two.
The columns of this diagram form the following
composition of $n$:
\begin{equation} \label{eqn:composition}
\mu_{\mu_1}^\bt,\mu_{\mu_1-1}^\bt,\cdots,\mu_2^\bt,\mu_1^\bt,\nu_1^\bt,\nu_2^\bt,
\cdots,\nu_{\nu_1-1}^\bt,\nu_{\nu_1}^\bt.
\end{equation}
For example, we represent
\[ 
((3^2 1^3);(2^3))\quad\text{ as }\quad
\begin{tableau}
\row{\c\c\rc\c\c}
\row{\c\c\rc\c\c}
\row{\q\q\rc\c\c}
\row{\q\q\rc}
\row{\q\q\rc}
\end{tableau}
\]
If we are dealing with $(v,x)\in\cO_{\mu;\nu}$ and 
have chosen a normal basis $\{v_{ij}\}$ for $(v,x)$, then
we can identify each basis element $v_{ij}$ with the $j$th box of the
$i$th row of the diagram, as in the proof of
Lemma \ref{lem:jordan-dist}. We then have that $v$ is the sum of the basis
elements in the column immediately to the left of the dividing line, and
Proposition \ref{prop:dim}(5) says that $E^x v$ is the span of all the
basis elements to the left of the dividing line (\textit{i.e.}, in the first
$\mu_1$ columns). More generally,
we define a partial flag 
\[
0=W_0^{(v,x)}\subset W_1^{(v,x)}\subset\cdots\subset
W_{\mu_1+\nu_1}^{(v,x)}=V
\] 
by the rule:
\begin{equation}
W_k^{(v,x)}=\begin{cases}
x^{\mu_1-k}E^x v,&\text{ if $k<\mu_1$,}\\
E^x v,&\text{ if $k=\mu_1$,}\\
(x^{k-\mu_1})^{-1}(E^x v),&\text{ if $k>\mu_1$.}
\end{cases}
\end{equation}
Clearly $W_k^{(v,x)}$ is the span of the basis elements in the first $k$ columns.

%
\section{Orbit Closures}
\label{sect:closures}

Our attention now turns to the Zariski closures of the
$G$-orbits in $V\times\cN$. Some easy facts (compare Lemma \ref{lem:easy}) are:

\begin{lem}
Let $(\mu;\nu)\in\cQ_n$, and let $\lambda=\mu+\nu$.
\begin{enumerate}
\item $\overline{\cO_{\mu;\nu}}\subseteq V\times\overline{\cO_\lambda}$.
\item $\overline{\cO_{\varnothing;\lambda}}=\{0\}\times\overline{\cO_\lambda}$.
\item $\overline{\cO_{\lambda;\varnothing}}=V\times\overline{\cO_\lambda}$.
\end{enumerate}
\end{lem}

\begin{proof}
Part (1) follows from the fact that $\cO_{\mu;\nu}\subseteq V\times\cO_\lambda$.
Part (2) is obvious.
In part (3), the inclusion $\subseteq$ is obvious, 
and the right-hand side is an irreducible variety of
the same dimension as the left-hand side.
\end{proof}

The closures $\overline{\cO_{\mu;\nu}}$ are in general singular varieties, and
our first aim is to define resolutions of their singularities. Motivated by
a standard construction for the closures $\overline{\cO_\lambda}$, we consider
partial flags whose successive codimensions are given by the composition
\eqref{eqn:composition}.

\begin{defn}
For any $(\mu;\nu)\in\cQ_n$, define a partial flag variety
\[
\begin{split} 
\cF_{\mu;\nu}=\{0=V_0 &\subset V_1\subset V_2\subset 
\cdots \subset V_{\mu_1+\nu_1}=V\,|\,\\
&\dim V_{\mu_1-i}=|\mu|-\mu_1^\bt-\cdots-\mu_i^\bt\text{ for }0\leq i\leq \mu_1,\\
&\dim V_{\mu_1+i}=|\mu|+\nu_1^\bt+\cdots+\nu_i^\bt\text{ for }0\leq i\leq \nu_1\},
\end{split}
\]
and two related varieties
\[
\begin{split}
\widehat{\cF_{\mu;\nu}}&=\{(x,(V_k))\in \cN\times \cF_{\mu;\nu}\,|\,
x(V_k)\subseteq V_{k-1}\text{ for }1\leq k\leq \mu_1+\nu_1\},\\
\widetilde{\cF_{\mu;\nu}}&=\{(v,x,(V_k))\in V\times\cN\times \cF_{\mu;\nu}\,|\,
v\in V_{\mu_1},\ (x,(V_k))\in \widehat{\cF_{\mu;\nu}}\}.
\end{split}
\]
We have obvious actions of $G$ on these varieties.
Let 
\[ \psi_{\mu;\nu}:\widehat{\cF_{\mu;\nu}}\to\cN\text{ and }
\pi_{\mu;\nu}:\widetilde{\cF_{\mu;\nu}}\to V\times\cN \]
denote the projection maps,
which are $G$-equivariant.
\end{defn}

The statements relating to $\psi_{\mu;\nu}$ in the following result
are known, but included for ease of reference; they hold for
general compositions of $n$, not just those of the form
\eqref{eqn:composition}, but this is
the only case we will need.

\begin{prop} \label{prop:resolution}
For any $(\mu;\nu)\in\cQ_n$,
\begin{enumerate}
\item the varieties
$\widehat{\cF_{\mu;\nu}}$ and
$\widetilde{\cF_{\mu;\nu}}$ are nonsingular and irreducible;
\item the projections 
$\psi_{\mu;\nu}:\widehat{\cF_{\mu;\nu}}\to\cN$ and
$\pi_{\mu;\nu}:\widetilde{\cF_{\mu;\nu}}\to V\times\cN$ are proper;
\item the image of $\psi_{\mu;\nu}:\widehat{\cF_{\mu;\nu}}\to \cN$ is 
the closure $\overline{\cO_{\mu+\nu}}$;
\item the restriction of $\psi_{\mu;\nu}$ to 
$\psi_{\mu;\nu}^{-1}(\cO_{\mu+\nu})$ is an isomorphism
onto $\cO_{\mu+\nu}$;
\item the image of $\pi_{\mu;\nu}:\widetilde{\cF_{\mu;\nu}}\to V\times\cN$ 
is the closure $\overline{\cO_{\mu;\nu}}$; and
\item the restriction of $\pi_{\mu;\nu}$ to 
$\pi_{\mu;\nu}^{-1}(\cO_{\mu;\nu})$ is an isomorphism
onto $\cO_{\mu;\nu}$.
\end{enumerate}
In summary, 
$\psi_{\mu;\nu}:\widehat{\cF_{\mu;\nu}}\to\overline{\cO_{\mu+\nu}}$ is a resolution 
of singularities of $\overline{\cO_{\mu+\nu}}$, and
$\pi_{\mu;\nu}:\widetilde{\cF_{\mu;\nu}}\to\overline{\cO_{\mu;\nu}}$ is a resolution 
of singularities of $\overline{\cO_{\mu;\nu}}$.
\end{prop}

\begin{proof}
The partial flag variety $\cF_{\mu;\nu}$ is 
a homogenous variety for $G$ with parabolic stabilizers.
Let $P_{\mu;\nu}$ denote one of these stabilizers, say the stabilizer of
the partial flag $(V_k^{0})\in\cF_{\mu;\nu}$, and let $U_{\mu;\nu}$
be the unipotent radical of $P_{\mu;\nu}$.
Then we have an isomorphism
\begin{equation} 
P_{\mu;\nu}/U_{\mu;\nu}\cong
  \GL_{\mu_1^\bt}\times \cdots \times \GL_{\mu_{\mu_1}^\bt}\times
  \GL_{\nu_1^\bt}\times \cdots \times \GL_{\nu_{\nu_1}^\bt}.
\end{equation}
As a consequence, we have
\[ 
\begin{split}
\dim \cF_{\mu;\nu}&=\dim G-\dim P_{\mu;\nu}=\dim U_{\mu;\nu}=
\frac{\dim G-\dim(P_{\mu;\nu}/U_{\mu;\nu})}{2}\\
&=\frac{n^2-(\mu_1^\bt)^2-\cdots-(\mu_{\mu_1}^\bt)^2
			-(\nu_1^\bt)^2-\cdots-(\nu_{\nu_1}^\bt)^2}{2}\\
&=\frac{n^2-n}{2}-n(\mu)-n(\nu).
\end{split}
\]
The projection $\widehat{\cF_{\mu;\nu}}\to\cF_{\mu;\nu}$ is well known
to be a vector bundle; the fibre over $(V_k^0)$ is exactly
$\Lie(U_{\mu;\nu})$, and a common notation for $\widehat{\cF_{\mu;\nu}}$
is $G\times_{P_{\mu;\nu}}\Lie(U_{\mu;\nu})$.
In particular, we have
\begin{equation} \label{eqn:firstdim}
\dim \widehat{\cF_{\mu;\nu}}=2\dim\cF_{\mu;\nu}=n^2-n-2n(\mu)-2n(\nu)
=\dim\cO_{\mu+\nu}.
\end{equation}
Clearly the projection $\widetilde{\cF_{\mu;\nu}}\to\widehat{\cF_{\mu;\nu}}$
is also a vector bundle, of rank $|\mu|$ since the fibre over
$(x,(V_k))$ is just $V_{\mu_1}$. So
\begin{equation} \label{eqn:seconddim}
\dim \widetilde{\cF_{\mu;\nu}}=n^2-|\nu|-2n(\mu)-2n(\nu)=\dim\cO_{\mu;\nu},
\end{equation}
where the last equality uses Proposition \ref{prop:dim}(8). Since the total space
of a vector bundle
over a nonsingular irreducible variety is nonsingular
and irreducible, part (1) is proved. 

Part (2) follows from the fact that $\cF_{\mu;\nu}$ is a projective variety, and
$\widehat{\cF_{\mu;\nu}}$ and $\widetilde{\cF_{\mu;\nu}}$ are closed
subvarieties of $\cN\times\cF_{\mu;\nu}$ and $V\times\cN\times\cF_{\mu;\nu}$
respectively.

It follows that the images of $\psi_{\mu;\nu}$ and $\pi_{\mu;\nu}$ are
$G$-stable irreducible closed subvarieties of $\cN$ and $V\times\cN$ respectively.
Since $G$ has finitely many orbits in $\cN$ and in $V\times\cN$, we can conclude that
both images are the closure of a single $G$-orbit. 
Moreover, $\cO_{\mu;\nu}$ is contained
in the image of $\pi_{\mu;\nu}$, since for any $(v,x)\in\cO_{\mu;\nu}$,
$(v,x,(W_k^{(v,x)}))\in\widetilde{\cF_{\mu;\nu}}$;
this also shows that $\cO_{\mu+\nu}$ is contained in the image of $\psi_{\mu;\nu}$.
Since we have the dimension equalities \eqref{eqn:firstdim} and \eqref{eqn:seconddim},
parts (3) and (5) follow.

Part (4) asserts that
for any $x\in\cO_{\mu+\nu}$ there is a unique pair
$(x,(W_k))\in\psi_{\mu;\nu}^{-1}(x)$, and the map
$\cO_{\mu+\nu}\to\cF_{\mu;\nu}:x\mapsto (W_k)$ is a morphism of varieties.
This statement is part of the theory of
Richardson orbits in $\Lie(G)$, for which see 
\cite[Theorem 5.2.3 and Corollary 5.2.4]{carter}, for example. 
There is a dense 
$P_{\mu;\nu}$-orbit $\cO$ in $\Lie(U_{\mu;\nu})$, and its $G$-saturation
is known to be $\cO_{\mu+\nu}$. If $x_0$ is a fixed element of $\cO$, then
its stabilizer $G^{x_0}$ is contained in $P_{\mu;\nu}$, and
the unique conjugate $P'$ of $P_{\mu;\nu}$ satisfying $x_0\in\Lie(U_{P'})$
is $P_{\mu;\nu}$ itself. Hence for an arbitrary element $gx_0\in\cO_{\mu+\nu}$,
the unique conjugate $P'$ of $P_{\mu;\nu}$ satisfying $gx_0\in\Lie(U_{P'})$
is $gP_{\mu;\nu}g^{-1}$, and the map $gx_0\mapsto gP_{\mu;\nu}$ is a morphism
of varieties from $\cO_{\mu+\nu}\cong G/G^{x_0}$ to $G/P_{\mu;\nu}$, as required.

For part (6) we need to show that for any $(v,x)\in\cO_{\mu;\nu}$ there is
a unique triple in the fibre $\pi_{\mu;\nu}^{-1}(v,x)$, namely $(v,x,(W_k^{(v,x)}))$, 
and moreover that the
map $(v,x)\mapsto(W_k^{(v,x)})$ is a morphism of varieties from $\cO_{\mu;\nu}$
to $\cF_{\mu;\nu}$. Both claims clearly follow from part (4).
\end{proof}

We can now give an alternative characterization of $\overline{\cO_{\mu;\nu}}$,
which should be compared with Corollary \ref{cor:extension}.

\begin{cor} \label{cor:criterion}
If $(v,x)\in V\times\cN$, then $(v,x)\in\overline{\cO_{\mu;\nu}}$ if and only
if there exists a $|\mu|$-dimensional subspace $W$ of $V$ such that:
\begin{enumerate}
\item $v\in W$,
\item $x(W)\subseteq W$,
\item the Jordan type $\mu'$ of $x|_W$ satisfies $\mu'\leq\mu$, and
\item the Jordan type $\nu'$ of $x|_{V/W}$ satisfies $\nu'\leq\nu$.
\end{enumerate}
\end{cor}

\begin{proof}
By part (5) of Proposition \ref{prop:resolution}, $(v,x)\in\overline{\cO_{\mu;\nu}}$ if and only if there exists a partial flag $(V_k)\in\cF_{\mu;\nu}$ such that
$(v,x,(V_k))\in\widetilde{\cF_{\mu;\nu}}$. Setting $W=V_{\mu_1}$, we see that this
is equivalent to the existence of a $|\mu|$-dimensional subspace $W$ of $V$
satisfying conditions (1), (2), and the following:
\begin{enumerate} 
\item[(3')] there is a partial flag $0=W_0 \subset W_1\subset
\cdots \subset W_{\mu_1}=W$ such that 
\[ x(W_k)\subseteq W_{k-1}\text{ and }\dim W_{k}=\mu_{\mu_1-k+1}^\bt+\cdots+
\mu_{\mu_1}^\bt,\text{ for $k=1,\cdots,\mu_1$;} \]
\item[(4')] there is a partial flag $0=U_0 \subset U_1\subset
\cdots \subset U_{\nu_1}=V/W$ such that 
\[ x(U_k)\subseteq U_{k-1}\text{ and }\dim U_{k}=\nu_{1}^\bt+\cdots+
\nu_{k}^\bt,\text{ for $k=1,\cdots,\nu_1$.} \]
\end{enumerate}
By the $\nu=\varnothing$
and $\mu=\varnothing$ special cases of Proposition 
\ref{prop:resolution}, the condition
(3') is equivalent to $x|_W\in\overline{\cO_\mu}$,
where $\cO_\mu$ denotes the $\GL(W)$-orbit of nilpotent endomorphisms of $W$
whose Jordan type is $\mu$, 
and (4') is equivalent to $x|_{V/W}\in\overline{\cO_\nu}$,
where $\cO_\nu$ is defined similarly.
Finally, the closure relation among nilpotent orbits for
the general linear group is well known to be given by the dominance order 
on partitions.
\end{proof}
\noindent
Beware that the existence of a $|\mu|$-dimensional subspace $W$ of $V$ such that
(1)--(4) hold with equality in (3) and (4) does not imply that
$(v,x)\in\cO_{\mu;\nu}$. (The criterion in Corollary \ref{cor:extension}
refers to the specific subspace $W=E^x v$.)

\begin{exam} \label{exam:salient}
Two salient examples when $n=4$
are as follows. Firstly, suppose that $(v,x)\in\cO_{(1^2);(2)}$, and let
$\{v_{11},v_{12},v_{13},v_{21}\}$ be a normal basis for $(v,x)$; we have
$v=v_{11}+v_{21}$. Then
$W=\Span\{v_{11},v_{12},v_{21}\}$ is a three-dimensional $x$-stable subspace
containing $v$, such that the Jordan type of $x|_W$ is $(21)$ and that of
$x|_{V/W}$ is $(1)$. By Corollary \ref{cor:criterion}, we have
$(v,x)\in\overline{\cO_{(21);(1)}}$, and hence 
$\cO_{(1^2);(2)}\subset\overline{\cO_{(21);(1)}}$.
Secondly, suppose that $(v,x)\in\cO_{(2^2);\varnothing}$, and let
$\{v_{11},v_{12},v_{21},v_{22}\}$ be a normal basis for $(v,x)$; we have
$v=v_{12}+v_{22}$. Then
$W=\Span\{v_{11},v_{12}+v_{22},v_{21}\}$ is a three-dimensional $x$-stable subspace
containing $v$, such that the Jordan type of $x|_W$ is $(21)$ and that of
$x|_{V/W}$ is $(1)$. By Corollary \ref{cor:criterion}, we have
$(v,x)\in\overline{\cO_{(21);(1)}}$, and hence
$\cO_{(2^2);\varnothing}\subset\overline{\cO_{(21);(1)}}$.
\end{exam}

We now define the partial order which, we will show, 
corresponds to the closure ordering on $G$-orbits
in $V\times\cN$.

\begin{defn} \label{defn:ineq}
For $(\rho;\sigma),(\mu;\nu)\in\cQ_n$, we say that
$(\rho;\sigma)\leq(\mu;\nu)$ if and only if the following
inequalities hold for all $k\geq 0$:
\[
\begin{split}
\rho_1+\sigma_1+\rho_2+\sigma_2+\cdots+\rho_k+\sigma_k
&\leq \mu_1+\nu_1+\mu_2+\nu_2+\cdots+\mu_k+\nu_k,\text{ and}\\
\rho_1+\sigma_1+\cdots+\rho_k+\sigma_k+\rho_{k+1}
&\leq \mu_1+\nu_1+\cdots+\mu_k+\nu_k+\mu_{k+1}.
\end{split}
\]
\end{defn}

This coincides
with the partial order used by
Shoji for his ``limit symbols'' with $e=2$ (see \cite{shoji:limit}).  
Note that the inequalities of the first kind simply say that
$\rho+\sigma\leq\mu+\nu$ for the dominance order.
Obviously $\rho\leq\mu$ and $\sigma\leq\nu$ together imply
$(\rho;\sigma)\leq(\mu;\nu)$, but the converse is false.

To clarify the partial order, we describe
its covering relations: for $(\rho;\sigma),(\mu;\nu)\in\cQ_n$,
we say that $(\mu;\nu)$ covers $(\rho;\sigma)$ if
$(\rho;\sigma)<(\mu;\nu)$ and there is no $(\tau;\upsilon)\in\cQ_n$
such that $(\rho;\sigma)<(\tau;\upsilon)<(\mu;\nu)$.

\begin{lem} \label{lem:covering}
For $(\rho;\sigma),(\mu;\nu)\in\cQ_n$,
$(\mu;\nu)$ covers $(\rho;\sigma)$ if and only if one of the following holds.
\begin{enumerate}
\item $\sigma=\nu$, and for some $\ell>k\geq 2$ we have
\[ 
\begin{split}
&\rho_k=\mu_k-1,\ \rho_\ell=\mu_\ell+1,\ \rho_i=\mu_i
\text{ for $i\neq k,\ell$,}\\
&\text{either $\ell=k+1$ or }
\mu_k-1=\mu_{k+1}=\cdots=\mu_{\ell-1}=\mu_\ell+1,\\
&\text{and }\nu_{k-1}=\nu_k=\cdots=\nu_{\ell}.
\end{split}
\]
\item $\rho=\mu$, and for some $\ell>k\geq 1$ we have
\[ 
\begin{split}
&\sigma_k=\nu_k-1,\ \sigma_\ell=\nu_\ell+1,\ \sigma_i=\nu_i
\text{ for $i\neq k,\ell$,}\\
&\text{either $\ell=k+1$ or }
\nu_k-1=\nu_{k+1}=\cdots=\nu_{\ell-1}=\nu_\ell+1,\\
&\text{and }\mu_{k}=\mu_{k+1}=\cdots=\mu_{\ell+1}.
\end{split}
\]
\item For some $\ell\geq k\geq 1$ we have
\[
\begin{split}
&\rho_i=\mu_i-1\text{ and }\sigma_i=\nu_i+1\text{ for $k\leq i\leq\ell$,}\\ 
&\rho_i=\mu_i\text{ and }\sigma_i=\nu_i\text{ for $i<k$ and $i>\ell$,}\\
&\mu_k=\mu_{k+1}=\cdots=\mu_\ell>\mu_{\ell+1},\\
&\text{and }\nu_{k-1}>\nu_{k}=\nu_{k+1}=\cdots=\nu_{\ell}\
\textup{(}\text{ignore $\nu_{k-1}>\nu_{k}$ if $k=1$}\textup{).}
\end{split}
\]
\item For some $\ell\geq k\geq 1$ we have
\[
\begin{split}
&\sigma_i=\nu_i-1\text{ and }\rho_{i+1}=\mu_{i+1}+1
\text{ for $k\leq i\leq\ell$,}\\ 
&\sigma_i=\nu_i\text{ and }\rho_{i+1}=\mu_{i+1}\text{ for $i<k$ and $i>\ell$, and
$\rho_1=\mu_1$,}\\ 
&\nu_k=\nu_{k+1}=\cdots=\nu_\ell>\nu_{\ell+1},\\
&\text{and }\mu_{k}>\mu_{k+1}=\cdots=\mu_{\ell+1}.
\end{split}
\]
\end{enumerate}
\end{lem}

\begin{proof}
As with the covering relations in $\cP_n$, the situation becomes clearer
if we think diagrammatically.
The relation $(\rho;\sigma)<(\mu;\nu)$ just says
that the composition $(\rho_1,\sigma_1,\rho_2,\sigma_2,\cdots)$ is dominated
by (and not equal to) the composition $(\mu_1,\nu_1,\mu_2,\nu_2,\cdots)$.
This is equivalent to saying that the diagram of $(\rho;\sigma)$ can be obtained from
that of $(\mu;\nu)$ by a sequence of moves of boxes, where at each step
we move an outside corner box to an inside corner which is either in a lower
row or on the right-hand end of the same row. At the start and end of such a
sequence, the boxes on either side of the dividing line form the shape of a
partition; $(\mu;\nu)$ covers $(\rho;\sigma)$ exactly when there is no such
sequence which can be broken into two sequences with this property
(in other words, for every such sequence of moves starting at 
$(\mu;\nu)$ and ending at $(\rho;\sigma)$, the intermediate shapes are not
diagrams of bipartitions). 

The four types of covering relations in the statement 
correspond to the following operations on diagrams.
In type (1), a single box moves down on the $\mu$ side of the dividing line, from an
outside corner to the first available inside corner, there being no inside or outside
corners on the $\nu$ side between these two positions:
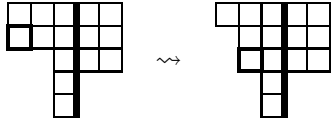
\[
\begin{tableau}
\row{\c\c\rc\c\c}
\row{\bc\c\rc\c\c}
\row{\q\q\rc\c\c}
\row{\q\q\rc}
\row{\q\q\rc}
\end{tableau}
\quad
\rightsquigarrow
\quad
\begin{tableau}
\row{\c\c\rc\c\c}
\row{\q\c\rc\c\c}
\row{\q\bc\rc\c\c}
\row{\q\q\rc}
\row{\q\q\rc}
\end{tableau}
\]
Type (2) is analogous, but with the box moving on the 
$\nu$ side of the dividing line:
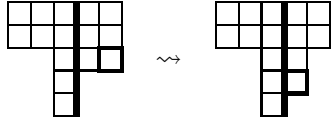
\[
\begin{tableau}
\row{\c\c\rc\c\c}
\row{\c\c\rc\c\c}
\row{\q\q\rc\c\bc}
\row{\q\q\rc}
\row{\q\q\rc}
\end{tableau}
\quad
\rightsquigarrow
\quad
\begin{tableau}
\row{\c\c\rc\c\c}
\row{\c\c\rc\c\c}
\row{\q\q\rc\c}
\row{\q\q\rc\bc}
\row{\q\q\rc}
\end{tableau}
\]
In type (3), a column of boxes (possibly a single box) moves directly to
the right, from an outside corner on the $\mu$ side to an inside corner
on the $\nu$ side:
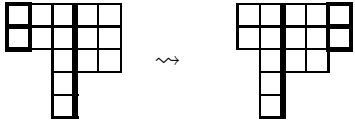
\[
\begin{tableau}
\row{\bc\c\rc\c\c}
\row{\bc\c\rc\c\c}
\row{\q\q\rc\c\c}
\row{\q\q\rc}
\row{\q\q\rc}
\end{tableau}
\quad
\rightsquigarrow
\quad
\begin{tableau}
\row{\q\c\rc\c\c\bc}
\row{\q\c\rc\c\c\bc}
\row{\q\q\rc\c\c}
\row{\q\q\rc}
\row{\q\q\rc}
\end{tableau}
\]
In type (4), a column of boxes (possibly a single box) moves to the left
and down one row, from an outside corner on the $\nu$ side to an inside corner
on the $\mu$ side:
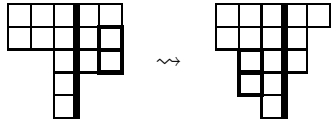
\[
\begin{tableau}
\row{\c\c\rc\c\c}
\row{\c\c\rc\c\bc}
\row{\q\q\rc\c\bc}
\row{\q\q\rc}
\row{\q\q\rc}
\end{tableau}
\quad
\rightsquigarrow
\quad
\begin{tableau}
\row{\c\c\rc\c\c}
\row{\c\c\rc\c}
\row{\q\bc\rc\c}
\row{\q\bc\rc}
\row{\q\q\rc}
\end{tableau}
\]
It is easy to see that none of these operations can be broken into two steps
while respecting the shape constraints.

Conversely, we must show that if $(\rho;\sigma) < (\mu;\nu)$, then 
we can apply one of these operations to $(\mu;\nu)$ to obtain a bipartition
$(\mu';\nu')$ which satisfies $(\mu';\nu')\geq(\rho;\sigma)$.
We will specify a suitable operation case by case, leaving the verification
that $(\mu';\nu')\geq(\rho;\sigma)$ to the reader.  

We can assume that the composition $(\rho_1,\sigma_1,\rho_2,\sigma_2,\cdots)$
first differs from the composition $(\mu_1,\nu_1,\mu_2,\nu_2,\cdots)$ in one of the
$\rho$--$\mu$ positions, since the alternative possibility can be reduced to this
by inserting a sufficiently large number $N$ at the start of both compositions
(such an insertion respects the partial order, and interchanges type (1) operations
with type (2), and type (3) with type (4)). Thus we have some
$i\geq 1$ such that $\rho_i<\mu_i$, and the first $i-1$ parts of $\rho$
(respectively, $\sigma$) equal those of $\mu$ (respectively, $\nu$).
Let $j$ be the largest integer such that $\mu_j = \mu_i$, and let
$j'$ be the largest integer such that $\nu_{j'}=\nu_i$ (or set
$j'=\infty$ if $\nu_i=0$). We now have six cases.

{\it Case I.} If $i = 1$, or if $i > 1$ and $\nu_{i-1} > \nu_j$, let 
$k\geq i$ be the smallest integer such that $\nu_k=\nu_j$,
and perform an operation of type~(3) with this $k$ and $\ell=j$.

{\it Case II.} If $i>1$ and $\nu_{i-1}=\nu_j>\nu_{j+1}$ (which forces $j'=j$),
perform an operation of type~(4) with $k=\ell=j$.

{\it Case III.} If $i > 1$, $\nu_{i-1}=\nu_{j+1}$ (which forces $j'\geq j+1$),
and $\mu_{j+1}\leq\mu_j-2$, perform an operation of type~(1) with $k=j$ and 
$\ell=j+1$.

{\it Case IV.} If $i > 1$, $\nu_{i-1}=\nu_{j+1}$, $\mu_{j+1}=\mu_j-1$,
$j'<\infty$, and $\mu_{j'+1}=\mu_j-1$, 
perform an operation of type~(4) with $k=j$ and $\ell=j'$.

{\it Case V.} If $i > 1$, $\nu_{i-1}=\nu_{j+1}$, $\mu_{j+1}=\mu_j-1$,
$j'<\infty$, and $\mu_j-1=\mu_{j'}>\mu_{j'+1}$, 
perform an operation of type~(4) with $k=\ell=j'$.

{\it Case VI.} If $i>1$, $\nu_{i-1}=\nu_{j+1}$, $\mu_{j+1}=\mu_j-1$,
and either $j'=\infty$ or $\mu_{j'}<\mu_j-1$, let $\ell$ be the smallest
integer such that $\mu_\ell<\mu_j-1$. Let $k$ be $\ell-1$ (if $\mu_\ell<\mu_j-2$)
or $j$ (if $\mu_\ell=\mu_j-2$). Perform an operation of type~(1) with
this $k$ and $\ell$.
This concludes the list of cases to be considered.
\end{proof}  


Notice that for fixed $\lambda\in\cP_n$, $\{(\mu;\nu)\in\cQ_n\,|\,\mu+\nu=\lambda\}$
is an interval in $\cQ_n$, in which all covering relations are of type (3).

\begin{exam}
Table~\ref{tbl:q4} shows the Hasse diagram of the poset $\cQ_4$.
The numbers in the left-hand
column are the dimensions of the corresponding orbits, and the labels
of the covering relations are the types from Lemma \ref{lem:covering}.
\end{exam}

\begin{table}
\[
\xymatrix@R=10pt@C=5pt{
16&
*+{\begin{tableau}\row{\c\c\c\rc}\end{tableau}} \ar@{-}[d]^3 \\
15&
*+{\begin{tableau}\row{\c\c\rc\c}\end{tableau}} \ar@{-}[d]^3\ar@{-}[drrr]^4 \\
14&
*+{\begin{tableau}\row{\c\rc\c\c}\end{tableau}} \ar@{-}[d]^3\ar@{-}[drrr]^4
&&& *+{\begin{tableau}\row{\c\c\rc}\row{\q\q\rc}\end{tableau}} \ar@{-}[d]^3 \\
13&
*+{\begin{tableau}\row{\rc\c\c\c}\end{tableau}} \ar@{-}[d]^3\ar@{-}[drr]^4 
&&& *+{\begin{tableau}\row{\c\rc\c}\row{\q\rc}\end{tableau}}
  \ar@{-}[dl]^3\ar@{-}[dr]_3\ar@{-}[drrr]^4 \\
12&
*+{\begin{tableau}\row{\lc\c\c\c}\end{tableau}} \ar@{-}[ddrrr]^2 
&& *+{\begin{tableau}\row{\rc\c\c}\row{\rc}\end{tableau}} \ar@{-}[dr]^3
&& *+{\begin{tableau}\row{\c\rc\c}\row{\q\q\lc}\end{tableau}}
  \ar@{-}[dl]_3\ar@{-}[ddrrrrr]^(.3){4}
&& *+{\begin{tableau}\row{\c\rc}\row{\c\rc}\end{tableau}}
  \ar@{-}[dd]^(.7){3}\ar@{-}[ddrrr]^1 \\
11&
&&& *+{\begin{tableau}\row{\rc\c\c}\row{\q\lc}\end{tableau}}
  \ar@{-}[d]^3\ar@{-}[drrr]^4 \\
10&
&&& *+{\begin{tableau}\row{\lc\c\c}\row{\lc}\end{tableau}} \ar@{-}[ddrrr]^2 
&&& *+{\begin{tableau}\row{\rc\c}\row{\rc\c}\end{tableau}}
  \ar@{-}[dd]^3\ar@{-}[drrr]^4
&&& *+{\begin{tableau}\row{\c\rc}\row{\q\rc}\row{\q\rc}\end{tableau}}
  \ar@{-}[d]^3 \\
9&
&&&&&&&&& *+{\begin{tableau}\row{\rc\c}\row{\rc}\row{\rc}\end{tableau}}
  \ar@{-}[dd]^3 \\
8&
&&&&&& *+{\begin{tableau}\row{\lc\c}\row{\lc\c}\end{tableau}} \ar@{-}[ddrrr]^2 \\
7&
&&&&&&&&& *+{\begin{tableau}\row{\rc\c}\row{\q\lc}\row{\q\lc}\end{tableau}}
  \ar@{-}[d]^3\ar@{-}[ddrrr]^4 \\
6&
&&&&&&&&& *+{\begin{tableau}\row{\lc\c}\row{\lc}\row{\lc}\end{tableau}}
  \ar@{-}[ddrrr]^2 \\
4&
&&&&&&&&&&&&
  *+{\begin{tableau}\row{\rc}\row{\rc}\row{\rc}\row{\rc}\end{tableau}}
  \ar@{-}[d]^3 \\
0&&&&&&&&&&&&&
  *+{\begin{tableau}\row{\lc}\row{\lc}\row{\lc}\row{\lc}\end{tableau}}
}
\]
\caption{Hasse diagram for $\cQ_4$.}\label{tbl:q4}
\end{table}

\begin{thm} \label{thm:closure}
For $(\rho;\sigma),(\mu;\nu)\in\cQ_n$, 
$\cO_{\rho;\sigma}\subseteq\overline{\cO_{\mu;\nu}}$ 
if and only if $(\rho;\sigma)\leq(\mu;\nu)$.
\end{thm}

\begin{proof}
We first prove the ``only if'' direction. Assume that
$\cO_{\rho;\sigma}\subseteq\overline{\cO_{\mu;\nu}}$. Since
$\overline{\cO_{\mu;\nu}}\subseteq V\times\overline{\cO_{\mu+\nu}}$,
we have
$\cO_{\rho+\sigma}\subseteq\overline{\cO_{\mu+\nu}}$, which as we know
implies the dominance condition
$\rho+\sigma\leq\mu+\nu$. All that remains is to prove the
inequalities of the
second kind, namely that for any $k\geq 0$,
\[ \rho_1+\sigma_1+\cdots+\rho_k+\sigma_k+\rho_{k+1}
\leq \mu_1+\nu_1+\cdots+\mu_k+\nu_k+\mu_{k+1}. \]
Our proof, like one of the standard proofs of the closure relation for ordinary
nilpotent orbits, rests on the fact that if $x\in\cN$ has Jordan type
$\lambda$, then $\lambda_1+\cdots+\lambda_k$ is the maximum possible dimension
of an $\F[x]$-submodule $\F[x]\{w_1,\cdots,w_k\}$ generated by $k$ elements
$w_1,\cdots,w_k$ of $V$. Thanks to Lemma \ref{lem:jordan-dist}, this implies that
for $(v,x)\in\cO_{\mu;\nu}$, $\mu_1+\nu_1+\cdots+\mu_k+\nu_k+\mu_{k+1}$ is the
maximum possible dimension of $\F[x]\{v,w_1,\cdots,w_k\}$ for
$w_1,\cdots,w_k\in V$. So the desired inequality amounts to saying that
for fixed $N$, the condition
\begin{equation} \label{eqn:dimcond}
\dim\F[x]\{v,w_1,\cdots,w_k\}\leq N\text{ for any $w_1,\cdots,w_k\in V$}
\end{equation}
is a closed condition on $(v,x)$ (\textit{i.e.}, it determines a closed subvariety
of $V\times\cN$). But no matter what $v,x,w_1,\cdots,w_k$ are,
$\F[x]\{v,w_1,\cdots,w_k\}$ is guaranteed to be spanned by the $(k+1)n$ vectors
\[ v,xv,\cdots,x^{n-1}v,w_1,xw_1,\cdots,x^{n-1}w_1,\cdots,w_k,xw_k,\cdots,
x^{n-1}w_k. \]
So $\dim\F[x]\{v,w_1,\cdots,w_k\}$ is the rank of the $n\times (k+1)n$
matrix which has these vectors as columns, and the condition \eqref{eqn:dimcond}
is equivalent to
the vanishing of all $(N+1)\times (N+1)$ minors of this matrix. This is
a collection of polynomial equations in the coordinates of $v$ and $x$ and the
indeterminate coordinates of $w_1\cdots,w_k$, so we are done.

To prove the ``if'' direction, we may assume that
$(\mu;\nu)$ covers $(\rho;\sigma)$, and invoke Lemma \ref{lem:covering}. Let
$(v,x)\in\cO_{\rho;\sigma}$, and let $\{v_{ij}\}$ be a normal basis for $(v,x)$.
To prove that $(v,x)\in\overline{\cO_{\mu;\nu}}$, it suffices to find a
$|\mu|$-dimensional subspace $W$ of $V$ satisfying conditions (1)--(4)
of Corollary \ref{cor:criterion}. Recall from Corollary \ref{cor:extension}
that $E^x v=\Span\{v_{ij}\,|\,1\leq i\leq \ell(\rho), j\leq\rho_i\}$
is $|\rho|$-dimensional, contains $v=\sum v_{i,\rho_i}$, 
and is preserved by $x$; moreover, the Jordan type of $x|_{E^x v}$ is $\rho$ 
and the Jordan type of $x|_{V/E^x v}$ is $\sigma$. Speaking rather loosely,
we will refer to the set $\{v_{ij}\,|\,j\leq\rho_i\}$ for fixed $i$
as the $i$th Jordan block of $x|_{E^x v}$, and the set
$\{v_{ij}\,|\,\rho_i<j\leq\rho_i+\sigma_i\}$ for fixed $i$ as the
$i$th Jordan block of $x|_{V/E^x v}$.

If the covering relation is one of the first
two types in Lemma \ref{lem:covering}, we simply take
$W=E^x v$. We have $|\rho|=|\mu|$, $\rho\leq\mu$ and $\sigma\leq\nu$,
so $E^x v$ meets all our requirements. In the other two types,
$E^x v$ must be modified slightly; the modifications we choose are modelled 
on Example \ref{exam:salient}.

In type (3), with $k$ and $\ell$ as
in Lemma \ref{lem:covering}, we take 
\[ W=\Span(\{v_{ij}\,|\,j\leq\rho_i\}\cup\{v_{k,\rho_k+1},v_{k+1,\rho_{k+1}+1},
\cdots,v_{\ell,\rho_\ell+1}\})\supseteq E^x v. \]
This is clearly
$|\mu|$-dimensional, contains $v$, and is preserved by $x$.
It is also obvious that the Jordan type of $x|_{W}$ is $\mu$,
since we have lengthened by $1$
the $k$th, $(k+1)$th, $\cdots$, and $\ell$th Jordan blocks of $x|_{E^x v}$.
Similarly, the Jordan type of $x|_{V/W}$ is $\nu$, since we have shortened by $1$
the corresponding Jordan blocks of $x|_{V/E^x v}$. 

In type (4), with $k$ and $\ell$ as
in Lemma \ref{lem:covering}, we take 
\[ 
\begin{split}
W=\Span((\{v_{ij}\,|\,j\leq\rho_i\}&\setminus\{v_{k,\rho_k},v_{k+1,\rho_{k+1}},
\cdots,v_{\ell+1,\rho_{\ell+1}}\})\\
&\cup\{v_{k,\rho_k}+v_{k+1,\rho_{k+1}}+\cdots+v_{\ell+1,\rho_{\ell+1}}\})
\subseteq E^x v. 
\end{split}
\] 
This is clearly
$|\mu|$-dimensional, contains $v$, and is preserved by $x$.
The Jordan type of $x|_V$ is $\mu$, because we have 
shortened by $1$ the $i$th Jordan block of $x|_{E^x v}$ for $k+1\leq i\leq\ell+1$,
and we have kept the $k$th Jordan block the same length but replaced
its generator $v_{k,\rho_k}$ by
$v_{k,\rho_k}+v_{k+1,\rho_{k+1}}+\cdots+v_{\ell+1,\rho_{\ell+1}}$.
Similarly, the Jordan type of $x|_{V/W}$ is $\nu$,
since we have lengthened by $1$
the $i$th Jordan block of $x|_{V/E^x v}$ for $k\leq i\leq\ell$,
and we have kept the $(\ell+1)$th Jordan block the same length
but replaced its generator $v_{\ell+1,\rho_{\ell+1}+\sigma_{\ell+1}}$
by $v_{k,\rho_k+\sigma_{\ell+1}}+v_{k+1,\rho_{k+1}+\sigma_{\ell+1}}+\cdots+
v_{\ell+1,\rho_{\ell+1}+\sigma_{\ell+1}}$.

So in all cases a suitable subspace $W$ can be found, and 
$(v,x)\in\overline{\cO_{\mu;\nu}}$ as required.
\end{proof}

\section{Fibres of the Resolutions of Singularities}
\label{sect:fibres}

For any $(\mu;\nu)\in\cQ_n$, we would like to describe
the intersection cohomology complex
$\IC(\overline{\cO_{\mu;\nu}},\Qlb)$, in particular
the dimensions of its stalks. In the case of ordinary
nilpotent orbit closures, the intersection cohomology is closely related to the
cohomology of the fibres of the resolutions 
$\psi_{\mu;\nu}:\widehat{\cF_{\mu;\nu}}\to\overline{\cO_{\mu+\nu}}$,
which are generalized Springer fibres of type A (``generalized'' in that
they involve the partial flag variety $\cF_{\mu;\nu}$ rather than the
complete flag variety; in the terminology of \cite{borhomacp}, they are examples
of Spaltenstein's varieties $\cP_x^0$). Analogously, we need to study the fibres of
$\pi_{\mu;\nu}:\widetilde{\cF_{\mu;\nu}}\to\overline{\cO_{\mu;\nu}}$.
We adopt a convenient abuse of notation for these fibres:
for $x\in\overline{\cO_{\mu+\nu}}$, $\psi_{\mu;\nu}^{-1}(x)$ will refer to the variety
of partial flags $(V_k)\in\cF_{\mu;\nu}$ such that 
$x(V_k)\subseteq V_{k-1}$ for $1\leq k\leq\mu_1+\nu_1$, not to the corresponding
variety of pairs $(x,(V_k))$; and similarly, we regard
$\pi_{\mu;\nu}^{-1}(v,x)$ as the closed subvariety of
$\psi_{\mu;\nu}^{-1}(x)$ defined by the extra condition $v\in V_{\mu_1}$.

Recall that the resolution $\psi_{\mu;\nu}$ is semismall in the sense
of Goresky and MacPherson. In fact, Spaltenstein in \cite{spaltenstein} proved
a more precise statement:
\begin{thm} \label{thm:spaltenstein}
Let $(\mu;\nu)\in\cQ_n$. For $x\in\cO_\pi\subseteq\overline{\cO_{\mu+\nu}}$,
$\psi_{\mu;\nu}^{-1}(x)$ has $K_{\pi^\bt(\mu+\nu)^\bt}$ irreducible
components, all of dimension
\[ n(\pi)-n(\mu+\nu)=\frac{\dim\cO_{\mu+\nu}-\dim\cO_\pi}{2}. \]
\end{thm}
\noindent
Here $K_{\pi^\bt(\mu+\nu)^\bt}$ is the Kostka number, defined
as in \cite[I.\S6]{macdonald}.

For $(v,x)\in V\times\cN$, let
$P^{(v,x)}$ denote the parabolic subgroup of $G$ which is the stabilizer
of the partial flag $(W_k^{(v,x)})$ defined in 
Section \ref{sect:param}.
Recall that $E^x v$ is one of the subspaces in this
partial flag, and that $x(W_k^{(v,x)})\subseteq W_{k-1}^{(v,x)}$, which means that
$x$ belongs to $\Lie(U^{(v,x)})$, where $U^{(v,x)}$ is the unipotent radical
of $P^{(v,x)}$. We can regard $(v,x)$ as an element of the vector space
$E^x v\oplus\Lie(U^{(v,x)})$, on which $P^{(v,x)}$ acts.

\begin{lem} \label{lem:density}
The $P^{(v,x)}$-orbit of 
$(v,x)$ is dense in $E^x v\oplus\Lie(U^{(v,x)})$.
\end{lem}

\begin{proof}
Recall from the proof of Proposition \ref{prop:resolution} that
the Richardson orbit of $P^{(v,x)}$ is the one containing $x$, so
the $P^{(v,x)}$-orbit of $x$ is dense in $\Lie(U^{(v,x)})$. Hence
it suffices to show that the $(G^x\cap P^{(v,x)})$-orbit of $v$ is dense in
$E^x v$. But $G^x\cap P^{(v,x)}$ is dense in $E^x\cap\Lie(P^{(v,x)})$,
and $(E^x\cap\Lie(P^{(v,x)}))v=E^x v$ because, in the notation of
Proposition \ref{prop:dim}, $v_{ij}=y_{i,i,\mu_i-j}v$ for all
$v_{ij}$ in the basis of $E^x v$.
\end{proof}

\begin{lem} \label{lem:dlp-dim}
Suppose that $(V_k^0)\in\pi_{\mu;\nu}^{-1}(v,x)$, and let
$\cO$ be the $P^{(v,x)}$-orbit of $(V_k^0)$ in $\cF_{\mu;\nu}$.
Let $P_{\mu;\nu}$ denote the stabilizer in $G$ of the partial flag
$(V_k^0)$, and $U_{\mu;\nu}$ its unipotent radical. Then
$\cO\cap\pi_{\mu;\nu}^{-1}(v,x)$ is a nonsingular locally closed subvariety
of $\pi_{\mu;\nu}^{-1}(v,x)$, of dimension
\[ 
\dim\left(\frac{P^{(v,x)}}{U^{(v,x)}}\right)
-\dim\left(\frac{P^{(v,x)}\cap P_{\mu;\nu}}
{U^{(v,x)}\cap U_{\mu;\nu}}\right)
-\dim\left(\frac{E^x v}{E^x v\cap V_{\mu_1}^0}\right). \]
\end{lem}

\begin{proof}
The variety $\cO\cap\pi_{\mu;\nu}^{-1}(v,x)$
is clearly isomorphic to
\[ \{p\in P^{(v,x)}\,|\,p^{-1}.(v,x)\in (E^x v\cap V_{\mu_1}^0)\oplus
(\Lie(U^{(v,x)})\cap\Lie(U_{\mu;\nu}))\}/(P^{(v,x)}\cap P_{\mu;\nu}), \]
so it suffices to prove that
\[ 
\{p\in P^{(v,x)}\,|\,p^{-1}.(v,x)\in (E^x v\cap V_{\mu_1}^0)\oplus
(\Lie(U^{(v,x)})\cap\Lie(U_{\mu;\nu}))\}
\]
is nonsingular and has dimension
\[ \dim P^{(v,x)}
-\dim (E^x v\oplus\Lie(U^{(v,x)}))
+\dim ((E^x v\cap V_{\mu_1}^0)\oplus(\Lie(U^{(v,x)})\cap\Lie(U_{\mu;\nu}))). 
\]
As is observed in a general context in \cite[Lemma 2.2]{dlp}, this is implied
by the density proved in the previous Lemma.
\end{proof}

\begin{prop} \label{prop:kostka}
Let $(v,x)\in\cO_{\rho;\sigma}\subseteq\overline{\cO_{\mu;\nu}}$.
Let $X$ be the closed subvariety of $\pi_{\mu;\nu}^{-1}(v,x)$ defined by
the extra condition $V_{\mu_1}=E^x v$.
\begin{enumerate}
\item $X$ is empty unless 
$\rho\leq\mu$ and $\sigma\leq\nu$, in which case
it has $K_{\rho^\bt\mu^\bt}K_{\sigma^\bt\nu^\bt}$ irreducible components,
all of dimension $n(\rho+\sigma)-n(\mu+\nu)$.
\item $\dim(\pi_{\mu;\nu}^{-1}(v,x)\setminus X)
<n(\rho+\sigma)-n(\mu+\nu)
+\frac{|\mu|-|\rho|}{2}$.
\end{enumerate}
\end{prop}

\begin{proof}
It is clear that $X$ is empty unless $\dim E^x v=|\mu|$, in which case
$X\cong\psi_{\mu;\varnothing}^{-1}(x|_{E^x v})\times
\psi_{\varnothing;\nu}^{-1}(x|_{V/E^x v})$, where we use
$E^x v$ and $V/E^x v$ in place of the vector space $V$ in defining
$\psi_{\mu;\varnothing}$ and $\psi_{\varnothing;\nu}$ respectively.
Recalling that $x|_{E^x v}\in\cO_\rho$ and $x|_{V/E^x v}\in\cO_\sigma$,
part (1) follows from Theorem \ref{thm:spaltenstein}.

To prove part (2), we observe that 
$\pi_{\mu;\nu}^{-1}(v,x)\setminus X$ is a union
of finitely many locally closed pieces $\cO\cap\pi_{\mu;\nu}^{-1}(v,x)$ as in
Lemma \ref{lem:dlp-dim}, because $P^{(v,x)}$ has finitely many orbits in 
$\cF_{\mu;\nu}$ (and fixes $E^x v$). So it suffices to show the desired inequality for
one of these pieces, where in addition to the dimension formula 
of Lemma \ref{lem:dlp-dim}
we know that $V_{\mu_1}^0\neq E^x v$. Since
$\dim E^x v=|\rho|$ and $\dim V_{\mu_1}^0=|\mu|$,
\[ \dim\left(\frac{E^x v}{E^x v\cap V_{\mu_1}^0}\right)>\frac{|\rho|-|\mu|}{2}. \]
Also, by the same argument as in the proof of Lemma \ref{lem:dlp-dim},
\[ \dim\left(\frac{P^{(v,x)}}{U^{(v,x)}}\right)
-\dim\left(\frac{P^{(v,x)}\cap P_{\mu;\nu}}
{U^{(v,x)}\cap U_{\mu;\nu}}\right) \]
is the dimension of a subvariety of $\psi_{\mu;\nu}^{-1}(x)$, and so is at most
$n(\rho+\sigma)-n(\mu+\nu)$ by Theorem \ref{thm:spaltenstein}.
The result follows.
\end{proof}

The special case of the next result where $\mu+\nu=(n)$ was proved independently
in \cite[(11)]{fgt}.
\begin{thm} \label{thm:semismall}
Let $(\mu;\nu)\in\cQ_n$.
\begin{enumerate}
\item The resolution of singularities
$\pi_{\mu;\nu}:\widetilde{\cF_{\mu;\nu}}\to\overline{\cO_{\mu;\nu}}$
is semismall.
\item We have an isomorphism of semisimple perverse sheaves:
\[ R(\pi_{\mu;\nu})_*\Qlb[\dim\cO_{\mu;\nu}]
\cong\bigoplus_{\substack{\rho\leq\mu\\\sigma\leq\nu}}
K_{\rho^\bt\mu^\bt}K_{\sigma^\bt\nu^\bt}\,
\IC(\overline{\cO_{\rho;\sigma}},\Qlb)[\dim\cO_{\rho;\sigma}], \]
where $mA$ denotes $A\oplus\cdots\oplus A$ \textup{(}$m$ copies\textup{)}.
\item For $(v,x)\in\overline{\cO_{\mu;\nu}}$, we have
\[ \dim H^i(\pi_{\mu;\nu}^{-1}(v,x),\Qlb)
=\negthickspace\negthickspace
\sum_{\substack{\rho\leq\mu\\\sigma\leq\nu\\\overline{\cO_{\rho;\sigma}}\ni (v,x)}}
\negthickspace\negthickspace
K_{\rho^\bt\mu^\bt}K_{\sigma^\bt\nu^\bt}
\dim \cH_{(v,x)}^{i-2(n(\rho+\sigma)-n(\mu+\nu))}
\IC(\overline{\cO_{\rho;\sigma}},\Qlb). \]
\end{enumerate}
\end{thm}

\begin{proof}
Part (1) asserts that for
$(v,x)\in\cO_{\rho;\sigma}\subseteq\overline{\cO_{\mu;\nu}}$,
we have 
\begin{equation} 
\dim\pi_{\mu;\nu}^{-1}(v,x)\leq\frac{\dim\cO_{\mu;\nu}-\dim\cO_{\rho;\sigma}}{2}.
\end{equation}
By Proposition \ref{prop:dim}, this upper bound is nothing but
$n(\rho+\sigma)-n(\mu+\nu)
+\frac{|\mu|-|\rho|}{2}$, so the inequality follows from Proposition
\ref{prop:kostka}. Semismallness of $\pi_{\mu;\nu}$ implies that
$R(\pi_{\mu;\nu})_*\Qlb[\dim\cO_{\mu;\nu}]$ is a semisimple perverse sheaf on
$\overline{\cO_{\mu;\nu}}$. By $G$-equivariance and the fact that
the stabilizers $G^{(v,x)}$ are connected (Proposition \ref{prop:dim}(7)),
we have
\begin{equation} \label{eqn:decomp}
R(\pi_{\mu;\nu})_*\Qlb[\dim\cO_{\mu;\nu}]
\cong\bigoplus_{(\tau;\upsilon)\leq(\mu;\nu)}
m_{(\tau;\upsilon)}^{(\mu;\nu)}\,
\IC(\overline{\cO_{\tau;\upsilon}},\Qlb)[\dim\cO_{\tau;\upsilon}]
\end{equation}
for some nonnegative integers $m_{(\tau;\upsilon)}^{(\mu;\nu)}$. Recall
that if $(\rho;\sigma)<(\tau;\upsilon)$, then
\begin{equation} 
\cH^i\IC(\overline{\cO_{\tau;\upsilon}},\Qlb)|_{\cO_{\rho;\sigma}}=0
\text{ for }i\geq\dim\cO_{\tau;\upsilon}-\dim\cO_{\rho;\sigma}.
\end{equation}
So taking the stalk of the $(-\dim\cO_{\rho;\sigma})$th cohomology sheaf
of both sides of \eqref{eqn:decomp} at $(v,x)\in\cO_{\rho;\sigma}$, we find
\begin{equation}
\dim H^{\dim\cO_{\mu;\nu}-\dim\cO_{\rho;\sigma}}(\pi^{-1}(v,x),\Qlb)
=m_{(\rho;\sigma)}^{(\mu;\nu)}.
\end{equation}
But Proposition \ref{prop:kostka} implies that
\[
\dim H^{\dim\cO_{\mu;\nu}-\dim\cO_{\rho;\sigma}}(\pi^{-1}(v,x),\Qlb)
=\begin{cases}
K_{\rho^\bt\mu^\bt}K_{\sigma^\bt\nu^\bt}
&\text{ if $\rho\leq\mu$ and $\sigma\leq\nu$,}\\
0&\text{ otherwise.}
\end{cases} 
\]
Part (2) follows, and part (3) is an immediate consequence.
\end{proof}

Part (2) of Theorem \ref{thm:semismall} implies that the perverse sheaves
$R(\pi_{\mu;\nu})_*\Qlb[\dim\cO_{\mu;\nu}]$ form a basis for the Grothendieck
group of $G$-equivariant perverse sheaves on $V\times\cN$, because the
transition matrix
$(K_{\rho^\bt\mu^\bt}K_{\sigma^\bt\nu^\bt})_{(\mu;\nu),(\rho;\sigma)}$
is unitriangular. In particular, the simple perverse sheaves
$\IC(\overline{\cO_{\rho;\sigma}},\Qlb)[\dim\cO_{\rho;\sigma}]$ are
the unique complexes satisfying Theorem \ref{thm:semismall}(2).
Similarly, Theorem \ref{thm:semismall} part (3) can be used to determine the
local intersection cohomologies $\dim\cH^i_{(v,x)}\IC(\overline{\cO_{\mu;\nu}},\Qlb)$,
if the Betti numbers of the fibres $\pi_{\mu;\nu}^{-1}(v,x)$ are known.

We can now obtain a sheaf-theoretic
analogue of Corollary \ref{cor:criterion}, using a construction
similar to Lusztig's definition of multiplication in geometric Hall algebras
(\cite[\S3]{lusztig:quivers} -- see \cite[\S4]{fgt} for a definition of
``Hall bimodule'' based on the same idea). We need
to keep track of dimensions in our notation, so we temporarily write
$\F^n$ instead of $V$ and $\cN_n$ instead of $\cN$. Define varieties
\[
\begin{split}
\cG_{m,n-m}&=\{(v,x,W)\,|\,v\in W\subseteq \F^n,\, \dim W=m,\, x\in\cN_n,\,
x(W)\subseteq W\},\\
\cH_{m,n-m}&=\{(v,x,W,\psi_1,\psi_2)\,|\,(v,x,W)\in\cG_{m,n-m},\\
&\qquad\qquad \psi_1:W\isomto\F^m,\,\psi_2:\F^n/W\isomto\F^{n-m}\}.
\end{split}
\]
We have obvious projection maps $\pi_{m,n-m}:\cG_{m,n-m}\to\F^n\times\cN_n$
and $q:\cH_{m,n-m}\to\cG_{m,n-m}$, as well as a map
\[ r:\cH_{m,n-m}\to\cN_m\times\cN_{n-m}:(v,x,W,\psi_1,\psi_2)\mapsto
(\psi_1(x|_W)\psi_1^{-1},\psi_2(x|_{\F^n/W})\psi_2^{-1}). \]
Since $r$ is a bundle projection with a nonsingular fibre of dimension
$n^2+m$, the pull-back 
\[ r^*(\IC(\overline{\cO_\mu},\Qlb)[\dim\cO_\mu]
\boxtimes\IC(\overline{\cO_\nu},\Qlb)[\dim\cO_\nu])[n^2+m] \] 
is a simple perverse
sheaf on $\cH_{m,n-m}$ for any $\mu\in\cP_m$, $\nu\in\cP_{n-m}$.
Since it is equivariant for the obvious $(GL(\F^m)\times GL(\F^{n-m}))$-action on
$\cH_{m,n-m}$ (of which $q$ is the quotient projection),
it must be isomorphic to 
$q^*A_{\mu;\nu}[m^2+(n-m)^2]$ for
some simple perverse sheaf $A_{\mu;\nu}$ on $\cG_{m,n-m}$.

\begin{prop} \label{prop:bimodule}
For any $(\mu;\nu)\in\cQ_n$, we have
\[ \IC(\overline{\cO_{\mu;\nu}},\Qlb)[\dim\cO_{\mu;\nu}]
\cong R(\pi_{|\mu|,|\nu|})_*\, A_{\mu;\nu}. \]
\end{prop}

\begin{proof}
We have a commutative diagram
\[
\xymatrix{
\widehat{\cF_{\mu;\varnothing}}\times\widehat{\cF_{\varnothing;\nu}} 
\ar[d]_{\psi_{\mu;\varnothing}\times\psi_{\varnothing;\nu}}
& X \ar[l] \ar[r] \ar[d] 
& \widetilde{\cF_{\mu;\nu}} \ar[d]^{\tilde{\pi}_{\mu;\nu}}\\
\cN_m\times\cN_{n-m} 
& \cH_{m,n-m} \ar[l]_(0.4){r} \ar[r]^{q}
& \cG_{m,n-m}
}
\]
where $\tilde{\pi}_{\mu;\nu}(v,x,(V_k))=(v,x,V_{\mu_1})$, and $X$ and the maps
emanating from it are defined so as to make both squares Cartesian.
Now Theorem \ref{thm:spaltenstein} implies that
\begin{equation} \label{eqn:littlekostka}
\begin{split}
R(\psi_{\mu;\varnothing})_*\Qlb[\dim\cO_\mu]
&\cong\bigoplus_{\rho\leq\mu}K_{\rho^\bt\mu^\bt}\,
\IC(\overline{\cO_\rho},\Qlb)[\dim\cO_\rho],\\
R(\psi_{\varnothing;\nu})_*\Qlb[\dim\cO_\nu]
&\cong\bigoplus_{\sigma\leq\nu}K_{\sigma^\bt\nu^\bt}\,
\IC(\overline{\cO_\sigma},\Qlb)[\dim\cO_\sigma].
\end{split}
\end{equation}
(See also \cite[Remark 5.7(3)]{hend:ft}.) Consequently,
\begin{equation}
R(\tilde{\pi}_{\mu;\nu})_*\Qlb[\dim\cO_\mu+\dim\cO_\nu+2|\mu||\nu|+|\mu|]
\cong\bigoplus_{\substack{\rho\leq\mu\\\sigma\leq\nu}}
K_{\rho^\bt\mu^\bt}K_{\sigma^\bt\nu^\bt}\,
A_{\rho;\sigma}.
\end{equation}
Applying $R(\pi_{|\mu|,|\nu|})_*$ to both sides, we obtain
\begin{equation} \label{eqn:bigkostka}
\qquad R(\pi_{\mu;\nu})_*\Qlb[\dim\cO_{\mu;\nu}]\\
\cong
\bigoplus_{\substack{\rho\leq\mu\\\sigma\leq\nu}}
K_{\rho^\bt\mu^\bt}K_{\sigma^\bt\nu^\bt}\,
R(\pi_{|\rho|,|\sigma|})_*A_{\rho;\sigma}.
\end{equation}
By the above-mentioned unitriangularity in 
Theorem \ref{thm:semismall}(2), the result follows.
\end{proof}
\noindent
This Proposition is essentially equivalent to \cite[Theorem 1]{fgt}.

Another known property of the generalized Springer fibre $\psi_{\mu;\nu}^{-1}(x)$
is that it has an \emph{affine paving} (an alpha-partition into
affine spaces, in the terminology of 
\cite[1.3]{dlp}), and consequently has no
odd-degree $\Qlb$-cohomology. The original proof, explained by Spaltenstein
in \cite[5.9]{spalt:book} in the Springer fibre case, is by induction on the
length of the partial flag, relying
on the fact that for fixed $x\in\cN$ and $\pi\in\cP_m$, a variety of the form
\[ \{W\subset V\,|\,\dim W=m,\, x(V)\subseteq W,\, x|_W\in\cO_{\pi}\} \]
can be paved by affine spaces (which in turn follows from the
fact that, under the constraint $x(V)\subseteq W$, 
the Jordan type of $x|_W$ is determined by that of $x$ and the
dimensions $\dim W\cap\ker(x^j)$). A naive analogue of this approach
for the fibres $\pi_{\mu;\nu}^{-1}(v,x)$ fails: for example, if
$(v,x)\in\cO_{(21^2);(1^3)}$, the variety
\[ \{W\subset V\,|\,\dim W=5, v\in W, x(V)\subseteq W, (v,x|_W)\in\cO_{(21^2);(1)}\}\]
is isomorphic to $\Aa^2\setminus\{0\}$. 
Hence the need to be somewhat more careful in proving:

\begin{thm} \label{thm:paving}
For any $(v,x)\in\overline{\cO_{\mu;\nu}}$, the
fibre $\pi_{\mu;\nu}^{-1}(v,x)$ has an affine paving.
\end{thm}

\begin{proof}
Let $P$ be the maximal parabolic subgroup of $G$ which
is the stabilizer of the subspace $\F[x] v$. We can partition 
$\cF_{\mu;\nu}$ into the orbits of $P$, which are well known
to be of the form
\[  (\cF_{\mu;\nu})_{(d_0,d_1,\cdots,d_{\mu_1+\nu_1})}
=\{(V_k)\in\cF_{\mu;\nu}\,|\,\dim(V_k\cap \F[x] v)=d_k,\, 0\leq k\leq\mu_1+\nu_1\}, \]
for integers $d_k$ which satisfy 
$0=d_0\leq d_1\leq \cdots\leq d_{\mu_1+\nu_1}=\dim \F[x] v$ (and also
some upper bounds on $d_{k+1}-d_k$
to guarantee non-emptiness of the above set, which need not concern us).
It suffices to show that each $\pi_{\mu;\nu}^{-1}(v,x)\cap(\cF_{\mu;\nu})_{(d_k)}$
has an affine paving. But for $(V_k)\in\pi_{\mu;\nu}^{-1}(v,x)$,
we have the extra information that $V_{\mu_1}\supseteq\F[x] v$, and that
$V_k\cap\F[x] v$ is $x$-stable. For $0\leq d\leq\dim\F[x]v$, let $U_d$ denote
the unique $d$-dimensional $x$-stable subspace of $\F[x] v$; so
for $(V_k)\in\pi_{\mu;\nu}^{-1}(v,x)$, the condition $\dim (V_k\cap \F[x] v)=d_k$
becomes $V_k\cap\F[x] v=U_{d_k}$. Moreover, the fact that 
$x(V_k)\subseteq V_{k-1}$ forces
$U_{d_k-1}=x(U_{d_k})\subseteq U_{d_{k-1}}$, so
$\pi_{\mu;\nu}^{-1}(v,x)\cap (\cF_{\mu;\nu})_{(d_k)}$ can only be nonempty when
\begin{equation} 
\begin{split}
&d_0=0,\\
&d_k=d_{k-1}\text{ or }d_{k-1}+1\text{ for }1\leq k\leq \mu_1,\\
&d_{\mu_1}=d_{\mu_1+1}=\cdots=d_{\mu_1+\nu_1}=\dim\F[x] v.
\end{split}
\end{equation}
Henceforth we fix integers $d_k$ satisfying these conditions.

We define a variety
\[
\begin{split}
Y=\{0=W_0&\subseteq W_1\subseteq \cdots \subseteq W_{\mu_1+\nu_1}=V/\F[x] v\,|\\
&\dim W_{\mu_1-i}=|\mu|-\mu_1^\bt-\cdots-\mu_i^\bt-d_{\mu_1-i}
\text{ for }0\leq i\leq \mu_1,\\
&\dim W_{\mu_1+i}=|\mu|+\nu_1^\bt+\cdots+\nu_i^\bt-d_{\mu_1+i}
\text{ for }0\leq i\leq \nu_1,\\
&x(W_k)\subseteq W_{k-1}\text{ for }1\leq k\leq\mu_1+\nu_1\}.
\end{split}
\]
The prescribed dimensions here are such that we have
a morphism 
\[ \Psi:\pi_{\mu;\nu}^{-1}(v,x)\cap (\cF_{\mu;\nu})_{(d_k)}\to Y:(V_k)\mapsto
((V_k+\F[x] v)/\F[x] v). \]
We clearly have an alpha-partition $Y=\bigcup_{(\tau_k)}Y_{(\tau_k)}$ where 
$(\tau_k)=(\tau_0,\tau_1,\cdots,\tau_{\mu_1+\nu_1})$ runs over sequences
of partitions where $|\tau_k|$ is the prescribed dimension of $W_k$, and
\[ Y_{(\tau_k)}=\{(W_k)\in Y\,|\,\text{Jordan type of }x|_{W_k}\text{ is }\tau_k\}. \]
(We need not go into the
conditions on $(\tau_k)$ which ensure that $Y_{(\tau_k)}$ is nonempty.)
Now $Y$ is a generalized Springer fibre based on the vector space $V/\F[x] v$,
although it is not quite of the form $\psi_{\tilde{\mu};\nu}^{-1}(x)$,
because the successive codimensions in the 
partial flag $W_0\subseteq W_1\subseteq \cdots\subseteq W_{\mu_1}$
need not be the columns of a Young diagram arranged in non-decreasing order.
Spaltenstein's argument still applies, however, and shows that each
$Y_{(\tau_k)}$ has an affine paving. 
Hence it suffices to show that the restriction of $\Psi$ to
$\Psi^{-1}(Y_{(\tau_k)})$ is a bundle projection with base
$Y_{(\tau_k)}$ and fibres isomorphic to affine space of some dimension.

Now fix $(W_k)\in Y$ and let $\widetilde{W}_k$ denote the preimage of $W_k$
under the projection $V\to V/\F[x] v$. The fibre $\Psi^{-1}((W_k))$ consists of all
partial flags $0=V_0\subset V_1\subset\cdots\subset V_{\mu_1+\nu_1}=V$ such that
\begin{equation} 
V_k+\F[x] v=\widetilde{W}_k,\ V_k\cap\F[x] v=U_{d_k},\ x(V_k)\subseteq V_{k-1}
\text{ for }1\leq k\leq\mu_1+\nu_1.
\end{equation}
Note that these conditions force $V_k=\widetilde{W}_k$ for $k\geq\mu_1$.
Imagine that $V_k$ is fixed and we are choosing $V_{k-1}$.
If $d_{k-1}=d_k$, then $V_{k-1}$ is forced to equal $V_k\cap\widetilde{W}_{k-1}$,
since this has the right dimension. If $d_{k-1}=d_k-1$, then $V_{k-1}$ must be
a codimension-$1$ subspace of $V_k\cap\widetilde{W}_{k-1}$ which contains
$x(V_k)+U_{d_{k-1}}$, and does not contain $x(V_k)+U_{d_{k}}$. But for any
vector spaces $A\subset B\subseteq C$ where $\dim(B/A)=1$, the variety
\[ \{D\subset C\,|\,\dim(C/D)=1,\, A\subseteq D,\, B\not\subseteq D\} \]
is isomorphic to affine space of dimension $\dim(C/B)$; in our case this
dimension is 
\[ \dim\left(\frac{V_k\cap\widetilde{W}_{k-1}}{x(V_k)+U_{d_{k}}}\right)
= (\dim W_{k-1}+d_k)-(\dim x(W_k)+d_k)=\dim(W_{k-1}/x(W_k)), \]
which is independent of $V_k$.
Since an affine space bundle over affine space is itself an affine space, we 
can conclude that
\begin{equation} 
\Psi^{-1}((W_k))\cong \Aa^{f(W_k)}\text{ where }
f(W_k)=\sum_{\substack{1\leq k\leq\mu_1\\d_{k-1}=d_k-1}}\dim(W_{k-1}/x(W_k)).
\end{equation}
The dimensions $\dim(W_{k-1}/x(W_k))$ are 
constant as $(W_k)$ runs over one of the $Y_{(\tau_k)}$
pieces, so the fibres $\Psi^{-1}((W_k))$ fit into a bundle as required.
\end{proof}

\begin{cor} \label{cor:oddzero}
Let $(\rho;\sigma),(\mu;\nu)\in\cQ_n$.
\begin{enumerate}
\item There is a polynomial $\Pi^{\rho;\sigma}_{\mu;\nu}(t)\in\N[t]$,
independent of $\F$, such that for any $(v,x)\in\cO_{\rho;\sigma}$,
\[
\begin{split}
&\sum_i \dim H^{2i}(\pi_{\mu;\nu}^{-1}(v,x),\Qlb)\,t^i=\Pi^{\rho;\sigma}_{\mu;\nu}(t),
\\
&\text{ and }H^i(\pi_{\mu;\nu}^{-1}(v,x),\Qlb)=0\text{ for $i$ odd.}
\end{split}
\]
\item There is a polynomial $\IC^{\rho;\sigma}_{\mu;\nu}(t)\in\N[t]$,
independent of $\F$, such that for any $(v,x)\in\cO_{\rho;\sigma}$,
\[
\begin{split}
&\sum_i \dim \cH^{2i}_{(v,x)}\IC(\overline{\cO_{\mu;\nu}},\Qlb)\,t^i=
\IC^{\rho;\sigma}_{\mu;\nu}(t),\\
&\text{ and }
\cH^i_{(v,x)}\IC(\overline{\cO_{\mu;\nu}},\Qlb)=0\text{ for $i$ odd.}
\end{split}
\]
\item
These polynomials are related by the rule:
\[ \Pi^{\tau;\upsilon}_{\mu;\nu}(t)
=\sum_{\substack{\rho\leq\mu\\\sigma\leq\nu\\(\rho;\sigma)\geq(\tau;\upsilon)}}
K_{\rho^\bt\mu^\bt}K_{\sigma^\bt\nu^\bt}\,t^{n(\rho+\sigma)-n(\mu+\nu)}
\IC^{\tau;\upsilon}_{\rho;\sigma}(t). \]
\item We have
\[
\begin{split}
&\Pi^{\rho;\sigma}_{\mu;\nu}(t)=\IC^{\rho;\sigma}_{\mu;\nu}(t)=0\text{ if }
(\rho;\sigma)\not\leq(\mu;\nu),\\
&\Pi^{\mu;\nu}_{\mu;\nu}(t)=\IC_{\mu;\nu}^{\mu;\nu}(t)=1,\\
&\Pi^{\rho;\sigma}_{\mu;\nu}(0)=\IC^{\rho;\sigma}_{\mu;\nu}(0)=1\text{ if }
(\rho;\sigma)<(\mu;\nu).\\
\end{split}
\]
\end{enumerate}
\end{cor}

\begin{proof}
For any variety $X$ with an affine paving, the
long exact sequence in cohomology with compact supports
shows that $H_c^i(X,\Qlb)=0$ for $i$ odd,
and that $\dim H_c^{2i}(X,\Qlb)$ is the number of spaces in the paving
which have dimension $i$. So part (1) is a consequence of Theorem \ref{thm:paving},
and the observation that the paving constructed in the proof does not
depend on the field $\F$. Parts (2) and (3) follow from part (1) 
via Theorem \ref{thm:semismall}(3). The only statement in part (4) which is
not automatic is that $\Pi^{\rho;\sigma}_{\mu;\nu}(0)=1$ if
$(\rho;\sigma)<(\mu;\nu)$, which is equivalent to saying that
the fibre $\pi_{\mu;\nu}^{-1}(v,x)$ is connected for all $(v,x)\in\cO_{\rho;\sigma}$.
This follows from part (3).
\end{proof}

\begin{exam}
Let $n=4$, and take $(v,x)\in\cO_{(1^3);(1)}$.
We will describe the affine paving of 
\[
\begin{split} 
\pi_{(3);(1)}^{-1}(v,x)=\{0=V_0\subset V_1&\subset V_2\subset V_3\subset V_4=V\,|\,\\
&\dim V_k=k,\, v\in V_3,\, x(V_k)\subseteq V_{k-1}\text{ for }1\leq k\leq 4\}
\end{split}
\]
given by the proof of Theorem \ref{thm:paving}. The nonempty pieces
$\pi_{(3);(1)}^{-1}(v,x)\cap(\cF_{(3);(1)})_{(d_k)}$ in this case are exactly
\[
\begin{split}
&\{(V_k)\in\pi_{(3);(1)}^{-1}(v,x)\,|\,v\in V_1\},\\
&\{(V_k)\in\pi_{(3);(1)}^{-1}(v,x)\,|\,v\in V_2,\,v\notin V_1\},\text{ and}\\
&\{(V_k)\in\pi_{(3);(1)}^{-1}(v,x)\,|\,v\in V_3,\,v\notin V_2\}.
\end{split}
\]
The first of these pieces is isomorphic via the appropriate map $\Psi$ to
\[ \{0=W_0=W_1\subset W_2\subset W_3\subset V/\F v\,|\,x(W_k)\subseteq W_{k-1}\}, \]
which is a Springer fibre of type $(21)$ (that is, a union of two projective lines
intersecting at a point). For the second of these pieces, the map $\Psi$ is
an $\Aa^1$-bundle with base
\[ \{0=W_0\subset W_1=W_2\subset W_3\subset V/\F v\,|\,x(W_k)\subseteq W_{k-1}\}, \]
which is another Springer fibre of type $(21)$. For the third piece, the image of
$\Psi$ is
\[ 
Y=\{0=W_0\subset W_1\subset W_2=W_3\subset V/\F v\,|\,x(W_k)\subseteq W_{k-1}\}, 
\]
which is another Springer fibre of type $(21)$. According to the proof of 
Theorem \ref{thm:paving}, we should partition $Y$
into
\[ 
Y'=\{(W_k)\in Y\,|\,x(W_2)=0\}\text{ and }Y''=\{(W_k)\in Y\,|\,x(W_2)\neq 0\}.
\]
We have $Y'\cong\Pp^1$ and $Y''\cong\Aa^1$; the restriction of $\Psi$ to
$\Psi^{-1}(Y')$ is an $\Aa^2$-bundle, while the restriction of $\Psi$ to
$\Psi^{-1}(Y'')$ is an $\Aa^1$-bundle. Thus we have
\[ \Pi^{(1^3);(1)}_{(3);(1)}(t)
=(2t+1)+t(2t+1)+t^2(t+1)+t^2=t^3+4t^2+3t+1. \]
A similar but easier calculation shows that
$\Pi^{(1^3);(1)}_{(21);(1)}(t)
=t^2+2t+1$.
Using Corollary \ref{cor:oddzero}(3), we deduce that
$\IC^{(1^3);(1)}_{(21);(1)}(t)=2t+1$ and
$\IC^{(1^3);(1)}_{(3);(1)}(t)=t+1$.
\end{exam}
\section{Intersection Cohomology and Kostka Polynomials}
\label{sect:kostka}
A famous theorem of Lusztig relates the intersection cohomology of
ordinary nilpotent orbit closures of type A to Kostka polynomials
(and hence to the representation theory of the symmetric group).
Let $\tK_{\lambda\pi}(t)=t^{n(\pi)}K_{\lambda\pi}(t^{-1})$ denote the
(modified) Kostka polynomial -- see \cite[III.\S6--7]{macdonald}.
In the notation defined in Corollary \ref{cor:oddzero}, Lusztig's result
\cite[Theorem 2]{lusztig:greenpolys} becomes:
\begin{thm} \label{thm:lusztig}
For $\pi,\lambda\in\cP_n$, 
$t^{n(\lambda)}\IC^{\varnothing;\pi}_{\varnothing;\lambda}(t)=\tK_{\lambda\pi}(t)$.
\end{thm}
\noindent
Note that this also implies
\begin{equation}
t^{n(\lambda)}\IC^{\rho;\sigma}_{\lambda;\varnothing}(t)=\tK_{\lambda,\rho+\sigma}(t),
\end{equation}
since $\IC(V\times\overline{\cO_\lambda},\Qlb)\cong(\Qlb)_V\boxtimes
\IC(\overline{\cO_\lambda},\Qlb)$.

In \cite{shoji:green} and \cite{shoji:limit},
Shoji has defined Kostka polynomials 
$\tK_{(\mu;\nu),(\rho;\sigma)}(t)\in\Z[t]$
which are indexed by pairs of bipartitions
rather than pairs of partitions (see especially \cite[Proposition 3.3]{shoji:limit},
where it is proved that these apparently rational functions are indeed polynomials).
The aim of this section is to prove the following analogue of
Theorem \ref{thm:lusztig}:
\begin{thm} \label{thm:main}
For $(\rho;\sigma),(\mu;\nu)\in\cQ_n$,
$t^{b(\mu;\nu)}\IC^{\rho;\sigma}_{\mu;\nu}(t^2)=\tK_{(\mu;\nu),(\rho;\sigma)}(t)$.
\end{thm}
\noindent
Theorem \ref{thm:main} immediately implies the following properties of 
Shoji's polynomials, not proved in \cite{shoji:limit}:
\begin{cor} \label{cor:shoji}
\begin{enumerate}
\item The coefficients in $\tK_{(\mu;\nu),(\rho;\sigma)}(t)$ are all nonnegative,
and only powers of $t$ which are congruent to $b(\mu;\nu)$ modulo $2$ occur.
\item
$\tK_{(\varnothing;\lambda),(\varnothing;\pi)}(t)=t^{|\lambda|}\tK_{\lambda\pi}(t^2)$,
and $\tK_{(\lambda;\varnothing),(\rho;\sigma)}(t)=\tK_{\lambda,\rho+\sigma}(t^2)$.
\end{enumerate}
\end{cor}

The defining property of Shoji's polynomials involves the representation theory
of the Coxeter group $W_n=W(\B_n)$. Recall that the set of irreducible characters
$\Irr(W_n)$ is naturally in bijection with $\cQ_n$ (see \cite[I.B.\S9]{macdonald}
or \cite[5.5]{gp}); we write $\chi^{\mu;\nu}$ for the character labelled by
$(\mu;\nu)$. The \emph{fake degree} $R(\chi)$ of a character $\chi$ of $W_n$
(not necessarily irreducible) is defined by
\begin{equation} 
R(\chi)=\frac{1}{2^n n!}\sum_{w\in W_n}\frac{\chi(w)\epsilon(w)
\prod_{a=1}^n (t^{2a}-1)}{\det(t-w)},
\end{equation}
where $\epsilon$ denotes the sign character of $W_n$ and $\det$ means the
determinant of the reflection representation. This fake degree is known to be a
nonnegative polynomial in the indeterminate $t$, because 
\cite[Proposition 11.1.1]{carter} implies that
\begin{equation} \label{eqn:fake}
R(\chi)=\sum_{i=0}^{n^2}\langle C^i(W_n),\chi\rangle_{W_n}\,t^i,
\end{equation}
where $C^i(W_n)$ is (the character of) the degree-$i$ homogeneous component
of the coinvariant algebra of $W_n$. We define a square matrix
\[ \Omega=(\omega_{(\mu;\nu),(\mu';\nu')})_{(\mu;\nu),(\mu';\nu')\in\cQ_n}
\text{ by }
\omega_{(\mu;\nu),(\mu';\nu')}=t^{n^2}R(\chi^{\mu;\nu}\otimes\chi^{\mu';\nu'}
\otimes\epsilon). \]
Then Shoji has proved (see \cite[Theorem 5.4]{shoji:green} 
and \cite[Remark 3.2]{shoji:limit}):
\begin{thm}\label{thm:ls}
There are unique
matrices $P = (p_{(\mu;\nu),(\rho;\sigma)})$ and 
$\Lambda =(\lambda_{(\rho;\sigma),(\tau;\upsilon)})$ over 
$\Q(t)$ satisfying the equation
$P \Lambda P^\bt = \Omega$
and subject to the following additional conditions:
\begin{equation*}
p_{(\mu;\nu),(\rho;\sigma)} = 
\begin{cases}
0 & \text{if $(\rho;\sigma)\not\leq(\mu;\nu)$,} \\
t^{b(\mu;\nu)} & \text{if $(\rho;\sigma)=(\mu;\nu)$,} 
\end{cases}
\qquad
\lambda_{(\rho;\sigma),(\tau;\upsilon)} = 0\quad
\text{if $(\rho;\sigma)\neq(\tau;\upsilon)$.}
\end{equation*}
The entry $p_{(\mu;\nu),(\rho;\sigma)}$ of the unique $P$
is $\tK_{(\mu;\nu),(\rho;\sigma)}(t)$.
\end{thm}
\noindent
The proof consists primarily of an algorithm for computing $P$ and $\Lambda$,
a special case of the ``generalized Lusztig--Shoji algorithm'' 
(see~\cite[Proposition~2.2]{gm} and the discussion in \cite[Section 2]{aa}).
Consequently, the uniqueness
statement still holds if the matrices $P$ and $\Lambda$ are assumed
to have entries in an extension field $K$ of $\Q(t)$. So to prove 
Theorem \ref{thm:main}, it suffices to find such a field $K$ and elements 
$\lambda_{(\tau;\upsilon)}\in K$, $(\tau;\upsilon)\in\cQ_n$, such that
\begin{equation} \label{eqn:innerprod}
\sum_{(\tau;\upsilon)\in\cQ_n}\lambda_{(\tau;\upsilon)}t^{b(\mu;\nu)+b(\mu';\nu')}
\IC^{\tau;\upsilon}_{\mu;\nu}(t^2)\IC^{\tau;\upsilon}_{\mu';\nu'}(t^2)
=\omega_{(\mu;\nu),(\mu';\nu')},
\end{equation}
for all $(\mu;\nu),(\mu';\nu')\in\cQ_n$. (It then follows
that in fact $\lambda_{(\tau;\upsilon)}\in\Q(t)$.)
Using Corollary \ref{cor:oddzero}(3), we see that \eqref{eqn:innerprod} is 
equivalent to
\begin{equation} \label{eqn:target}
\begin{split}
\sum_{(\tau;\upsilon)\in\cQ_n}&\lambda_{(\tau;\upsilon)}
\Pi^{\tau;\upsilon}_{\mu;\nu}(t^2)\Pi^{\tau;\upsilon}_{\mu';\nu'}(t^2)\\
&=t^{-b(\mu;\nu)-b(\mu';\nu')}
\sum_{\substack{\rho\leq\mu\\\rho'\leq\mu'\\\sigma\leq\nu\\\sigma'\leq\nu'}}
K_{\rho^\bt\mu^\bt}K_{\sigma^\bt\nu^\bt}
K_{(\rho')^\bt(\mu')^\bt}K_{(\sigma')^\bt(\nu')^\bt}\,
\omega_{(\rho;\sigma),(\rho';\sigma')},
\end{split}
\end{equation}
for all $(\mu;\nu),(\mu';\nu')\in\cQ_n$, and this is the form we will prove.

We first want to simplify the right-hand side. Interpret $W_n$ as the group of
permutations of $\{\pm 1,\pm 2,\cdots,\pm n\}$ which commute with 
$i\leftrightarrow -i$. For any composition
$n_1,n_2,\cdots,n_k$ of $n$, let $W_{(n_i)}$ denote the subgroup
$W_{n_1}\times W_{n_2}\times \cdots \times W_{n_k}$ of $W_n$, the preimage
of the Young subgroup $S_{(n_i)}=S_{n_1}\times S_{n_2}\times \cdots \times S_{n_k}$
under the natural homomorphism $W_n\to S_n$. Given two compositions
$(n_i)_{i=1}^k$ and $(n_j')_{j=1}^{k'}$, the double cosets
$W_{(n_j')}\!\setminus\! W_n\,/\,W_{(n_i)}$ are clearly in bijection with the
double cosets $S_{(n_j')}\!\setminus\! S_n\,/\,S_{(n_i)}$. These in turn
(see \cite[\S2]{hend:icnilp}, for example) are in bijection with
$M_{(n_i),(n_j')}$, the set of $k\times k'$ matrices $(m_{ij})$ satisfying:
\begin{enumerate}
\item $m_{ij}\in\N$, for all $i,j$,
\item $\sum_{j}m_{ij}=n_i$, for all $i$,\text{ and } 
\item $\sum_{i}m_{ij}=n_j'$, for all $j$.
\end{enumerate}
Write $m_{\leq i,\leq j}$ for $\sum_{i'\leq i,j'\leq j}m_{i'j'}$,
and similarly define $m_{<i,<j}$ and $m_{<i,>j}$.
The bijection $S_{(n_j')}\!\setminus\! S_n\,/\,S_{(n_i)}\leftrightarrow 
M_{(n_i),(n_j')}$
is such that the double coset containing $w$ corresponds to the matrix $(m_{ij})$
which satisfies
\begin{equation}
m_{\leq i,\leq j}=
|\{s\leq n_1+\cdots+n_i\,|\,
w(s)\leq n_1'+\cdots+n_j'\}|.
\end{equation}
Given a bipartition $(\mu;\nu)\in\cQ_n$, we let 
$W_{(\mu;\nu)}$ be the subgroup of $W_n$ determined by
the composition \eqref{eqn:composition}; 
given a second bipartition $(\mu';\nu')$, write 
$M_{(\mu;\nu),(\mu';\nu')}$ for the set of matrices determined by the two
compositions.
\begin{prop} \label{prop:reps}
The right-hand side of \eqref{eqn:target} equals
\[
\sum_{(m_{ij})\in M_{(\mu;\nu),(\mu';\nu')}}
\frac{t^{2(\binom{n}{2}-n(\mu+\nu)-n(\mu'+\nu')+\sum_{i,j}\binom{m_{ij}}{2}
+m_{\leq\mu_1,\leq\mu_1'})}
\prod_{a=1}^n(t^{2a}-1)}
{\prod_{i,j} \prod_{a=1}^{m_{ij}}(t^{2a}-1)}.
\]
\end{prop}
\begin{proof}
Recall that $\chi^{\rho;\sigma}\otimes\epsilon=\chi^{\sigma^\bt;\rho^\bt}$,
so we have
\[ \omega_{(\rho;\sigma),(\rho';\sigma')}
=t^{n^2}R(\chi^{\sigma^\bt;\rho^\bt}\otimes
\chi^{(\sigma')^\bt;(\rho')^\bt}\otimes\epsilon). \]
Also $\chi^{\sigma^\bt;\rho^\bt}=
\Ind_{W_{|\sigma|}\times W_{|\rho|}}^{W_n}(\chi^{\sigma^\bt}
\boxtimes\delta\chi^{\rho^\bt})$, where $\chi^\lambda$ denotes
the irreducible character of $S_{|\lambda|}$ indexed by $\lambda$ and also
its pull-back to $W_{|\lambda|}$, and $\delta$ is the one-dimensional
character of $W_n$ such that $\delta\epsilon$ is the pull-back of the sign 
character of $S_n$ (and we continue to write $\delta$ for 
its restriction to any subgroup). Using the well-known fact that the Kostka
numbers give the multiplicities of irreducible characters of the symmetric group
in the inductions of the trivial character from Young subgroups, we find
\[
\begin{split}
\sum_{\substack{\rho\leq\mu\\\sigma\leq\nu}}
K_{\rho^\bt\mu^\bt}K_{\sigma^\bt\nu^\bt}\,\chi^{\sigma^\bt;\rho^\bt}
&=\Ind_{W_{|\nu|}\times W_{|\mu|}}^{W_n}\left(
\sum_{\sigma^\bt\geq\nu^\bt}K_{\sigma^\bt\nu^\bt}\chi^{\sigma^\bt}
\boxtimes \delta
\sum_{\rho^\bt\geq\mu^\bt}K_{\rho^\bt\mu^\bt}\chi^{\rho^\bt}
\right)\\
&=\Ind_{W_{|\nu|}\times W_{|\mu|}}^{W_n}\left(
\Ind^{W_{|\nu|}}_{W_{\nu_1^\bt}\times\cdots\times W_{\nu_{\nu_1}^\bt}}(1)
\boxtimes
\Ind^{W_{|\mu|}}_{W_{\mu_1^\bt}\times\cdots\times W_{\mu_{\mu_1}^\bt}}(\delta)
\right)\\
&=\Ind_{W_{\mu;\nu}}^{W_n}(\delta_{\mu;\nu}),
\end{split}
\]
where $\delta_{\mu;\nu}$ is the one-dimensional character of $W_{\mu;\nu}$
which is $\delta$ on all the $W_{\mu_i^\bt}$ components and trivial on all
the $W_{\nu_i^\bt}$ components. So the right-hand side of \eqref{eqn:target}
equals:
\[
\begin{split}
&t^{n^2-b(\mu;\nu)-b(\mu';\nu')}R(\Ind_{W_{\mu;\nu}}^{W_n}(\delta_{\mu;\nu})
\otimes \Ind_{W_{\mu';\nu'}}^{W_n}(\delta_{\mu';\nu'})\otimes\epsilon)\\
&=\frac{t^{n^2-b(\mu;\nu)-b(\mu';\nu')}\prod_{a=1}^n(t^{2a}-1)}{|W_n|}
\sum_{w\in W_n}\frac{\Ind_{W_{\mu;\nu}}^{W_n}(\delta_{\mu;\nu})(w)
\Ind_{W_{\mu';\nu'}}^{W_n}(\delta_{\mu';\nu'})(w)}{\det(t-w)}\\
&=\frac{t^{n^2-b(\mu;\nu)-b(\mu';\nu')}\prod_{a=1}^n(t^{2a}-1)}
{|W_n||W_{\mu;\nu}||W_{\mu';\nu'}|}
\sum_{\substack{w,w_1,w_2\in W_n\\w_1ww_1^{-1}\in W_{\mu;\nu}\\
w_2ww_2^{-1}\in W_{\mu';\nu'}}}
\frac{\delta_{\mu;\nu}(w_1ww_1^{-1})\delta_{\mu';\nu'}(w_2ww_2^{-1})}{\det(t-w)}\\
&=\frac{t^{n^2-b(\mu;\nu)-b(\mu';\nu')}\prod_{a=1}^n(t^{2a}-1)}
{|W_{\mu;\nu}||W_{\mu';\nu'}|}
\sum_{\substack{\tilde{w}\in W_n\\y\in W_{\mu;\nu}\cap
\tilde{w}^{-1}W_{\mu';\nu'}\tilde{w}}}
\frac{\delta_{\mu;\nu}(y)\delta_{\mu';\nu'}(\tilde{w}y\tilde{w}^{-1})}{\det(t-y)},
\end{split}
\]
where the last step uses the change of variables $\tilde{w}=w_2w_1^{-1}$,
$y=w_1ww_1^{-1}$ (and $w_1$ becomes a free variable, cancelling the $|W_n|$
from the denominator). 

Now if the double coset of
$\tilde{w}$ corresponds to the matrix
$(m_{ij})\in M_{(\mu;\nu),(\mu';\nu')}$, then 
$W_{\mu;\nu}\cap \tilde{w}^{-1}W_{\mu';\nu'}\tilde{w}$ is a reflection subgroup
of $W_n$ isomorphic to $\prod_{i,j}W_{m_{ij}}$, in such a way that its character
$y\mapsto\delta_{\mu;\nu}(y)\delta_{\mu';\nu'}(\tilde{w}y\tilde{w}^{-1})$
corresponds to the character which is $\delta$ on the factors
$W_{m_{ij}}$ where $i>\mu_1$ and $j\leq\mu_1'$ or $i\leq\mu_1$ and $j>\mu_1'$,
and trivial on the other factors. Using the analogue of \eqref{eqn:fake}
and the fact that $\epsilon$ occurs in $C^{n^2}(W_n)$ and $\delta\epsilon$
in $C^{n^2-n}(W_n)$, we find that
\[
\begin{split}
&\frac{1}{|W_{\mu;\nu}\cap \tilde{w}^{-1}W_{\mu';\nu'}\tilde{w}|}
\sum_{y\in W_{\mu;\nu}\cap \tilde{w}^{-1}W_{\mu';\nu'}\tilde{w}}
\frac{\delta_{\mu;\nu}(y)\delta_{\mu';\nu'}(\tilde{w}y\tilde{w}^{-1})}{\det(t-y)}\\
&=\frac{\prod_{i\leq\mu_1,j\leq\mu_1'}t^{m_{ij}^2}
\prod_{i\leq\mu_1,j>\mu_1'}t^{m_{ij}^2-m_{ij}}
\prod_{i>\mu_1,j\leq\mu_1'}t^{m_{ij}^2-m_{ij}}
\prod_{i>\mu_1,j>\mu_1'}t^{m_{ij}^2}
}
{\prod_{i,j}\prod_{a=1}^{m_{ij}}(t^{2a}-1)}\\
&=\frac{t^{n-|\mu|-|\mu'|+2m_{\leq\mu_1,\leq\mu_1'}+2\sum_{i,j}\binom{m_{ij}}{2}}}
{\prod_{i,j}\prod_{a=1}^{m_{ij}}(t^{2a}-1)}.
\end{split}
\]
Substituting this in the above and using
$|W_{\mu';\nu'}\tilde{w}W_{\mu;\nu}|=\frac{|W_{\mu;\nu}||W_{\mu';\nu'}|}
{|W_{\mu;\nu}\cap \tilde{w}^{-1}W_{\mu';\nu'}\tilde{w}|}$, we obtain the
result.
\end{proof}

To analyse the left-hand side of \eqref{eqn:target}, we return to our
enhanced nilpotent cone $V\times\cN$, choosing the base field $\F$
to be an algebraic closure of the finite field $\F_q$, where $q$ is some
prime power. It is evident
from the results of Section \ref{sect:param} that each $G$-orbit
$\cO_{\mu;\nu}$ is defined over $\F_q$.
\begin{prop} \label{prop:fq}
For all $(\mu;\nu),(\mu';\nu')\in\cQ_n$, we have
\[ 
\begin{split}
&\sum_{(\tau;\upsilon)\in\cQ_n}|\cO_{\tau;\upsilon}(\F_q)|\,
\Pi^{\tau;\upsilon}_{\mu;\nu}(q)\Pi^{\tau;\upsilon}_{\mu';\nu'}(q)\\
&=\sum_{(m_{ij})\in M_{(\mu;\nu),(\mu';\nu')}}
\frac{q^{\binom{n}{2}-n(\mu+\nu)-n(\mu'+\nu')+\sum_{i,j}\binom{m_{ij}}{2}
+m_{\leq\mu_1,\leq\mu_1'}}
\prod_{a=1}^n(q^{a}-1)}
{\prod_{i,j} \prod_{a=1}^{m_{ij}}(q^{a}-1)}.
\end{split}
\]
\end{prop}
\begin{proof}
For any $(v,x)\in V(\F_q)\times\cN(\F_q)$, the alpha-partition of
$\pi_{\mu;\nu}^{-1}(v,x)$ defined in Theorem \ref{thm:paving} is defined
over $\F_q$, and hence
\begin{equation} \label{eqn:fibrepurity}
|\pi_{\mu;\nu}^{-1}(v,x)(\F_q)|=\Pi_{\mu;\nu}^{\tau;\upsilon}(q),\text{ where }
(v,x)\in\cO_{\tau;\upsilon}(\F_q).
\end{equation}
Hence the left-hand side of our desired equation is the number of $\F_q$-points of
the variety
\[
\begin{split}
Z=\{(v,x,(V_i),(V_j'))\in V\times\cN\times & \cF_{\mu;\nu}\times\cF_{\mu';\nu'}\,|\\
&(v,x,(V_i))\in\widetilde{\cF_{\mu;\nu}},\,
(v,x,(V_j'))\in\widetilde{\cF_{\mu';\nu'}}\}.
\end{split}
\]
Now the $G$-orbits in $\cF_{\mu;\nu}\times\cF_{\mu';\nu'}$ are in bijection with
$M_{(\mu;\nu),(\mu';\nu')}$: the orbit $\cO_{(m_{ij})}$ corresponding to a matrix
$(m_{ij})$ consists of pairs $((V_i),(V_j'))$ satisfying
\begin{equation} 
\dim (V_i\cap V_j')=m_{\leq i,\leq j},\text{ for all }i,j.
\end{equation}
So we have a partition $Z=\bigcup_{(m_{ij})\in M_{(\mu;\nu),(\mu';\nu')}}Z_{(m_{ij})}$,
where
\[
\begin{split} 
Z_{(m_{ij})}=\{(v,x,(V_i),(V_j'))&\,|\,((V_i),(V_j'))\in\cO_{(m_{ij})},\\
&v\in V_{\mu_1}\cap V_{\mu_1'},\,x\in\cN,\,
x(V_i)\subseteq V_{i-1},\, x(V_j')\subseteq V_{j-1}'\}. 
\end{split} 
\]
Hence $|Z(\F_q)|=\sum_{(m_{ij})\in M_{(\mu;\nu),(\mu';\nu')}}|Z_{(m_{ij})}(\F_q)|$,
and we need only show that $|Z_{(m_{ij})}(\F_q)|$ is given by the $(m_{ij})$
term of the right-hand side of our desired equation.
By standard methods, we compute
\begin{equation} 
|\cO_{(m_{ij})}(\F_q)|
=\frac{\prod_{a=1}^n(q^{a}-1)}{\prod_{i,j} \prod_{a=1}^{m_{ij}}(q^{a}-1)}
\prod_{i,j}
q^{m_{ij}m_{<i,>j}},
\end{equation}
where the fraction represents the number of ways of choosing
the partial flag $(V_i)$ and the images of each $V_i\cap V_j'$ in $V_i/V_{i-1}$,
and the other product represents the number of ways of choosing the
subspaces $V_j'$ themselves once these images are fixed.
For any $((V_i),(V_j'))\in\cO_{m_{ij}}(\F_q)$, we have
\begin{equation}
\begin{split}
&|(V_{\mu_1}\cap V_{\mu_1'})(\F_q)|=q^{m_{\leq\mu_1,\leq\mu_1'}},\\
&|\{x\in\cN(\F_q)\,|\,x(V_i)\subseteq V_{i-1},\, x(V_j')\subseteq V_{j-1}'\}|
=\prod_{i,j}
q^{m_{ij}m_{<i,<j}}.
\end{split}
\end{equation}
Finally, to reconcile the powers of $q$, note that
\[ 
\begin{split}
\sum_{i,j}m_{ij}m_{<i,>j}+m_{ij}m_{<i,<j}
&=\frac{1}{2}\sum_{\substack{i,j,i',j'\\i'\neq i,j'\neq j}}m_{ij}m_{i'j'}\\
&=\frac{1}{2}(n^2-\sum_i (\sum_j m_{ij})^2-\sum_j (\sum_i m_{ij})^2
+\sum_{i,j} m_{ij}^2)\\
&=\binom{n}{2}-n(\mu+\nu)-n(\mu'+\nu')+\sum_{i,j}\binom{m_{ij}}{2},
\end{split}
\]
as required.
\end{proof}

Now let $R$ be the ring of all functions $g: \Z_{>0} \to \Qlb$ of the form
\begin{equation}\label{eqn:Z-char}
g(s) = \sum_i c_i (a_i)^s
\qquad\text{with $c_i \in \Z$ and $a_i \in \Qlb$ (a finite sum).}
\end{equation}
By well-known facts, any $g \in R$ can be expressed in the
above form in a unique way, and $R$ is an integral domain.
We fix a square root $q^{1/2}$ of $q$ in $\Qlb$, and
identify $\Z[t]$ with a subring of $R$ via the map which sends
a polynomial $p(t)$ to the function $s\mapsto p(q^{s/2})$.
Let $K$ denote the fraction field of $R$, an extension field of $\Q(t)$.
It is easy to see that $\Q(t) \cap R = \Z[t,t^{-1}]$.

For any $(\tau;\upsilon)\in\cQ_n$, we define an element
$\lambda_{(\tau;\upsilon)} \in R$
by the rule
\[ \lambda_{(\tau;\upsilon)}(s)=
|\cO_{\rho;\sigma}(\F_{q^s})|=\sum_i(-1)^i\tr(F^s\,|\,H_c^i(\cO_{\rho;\sigma},\Qlb)),
\]
where $F$ denotes the Frobenius endomorphism of $\cO_{\rho;\sigma}$ relative
to $\F_q$. 
Comparing Proposition \ref{prop:reps} and Proposition
\ref{prop:fq} (with $q$ replaced by a general power $q^s$), we see that
equation \eqref{eqn:target} holds
in the field $K$. This completes the proof of
Theorem \ref{thm:main}.

We obtain the following result as a by-product of this proof.
\begin{prop} \label{prop:qpoints}
For any $(\tau;\upsilon)\in\cQ_n$, there is a polynomial
$\theta_{(\tau;\upsilon)}(t)\in\Z[t]$ such that
$|\cO_{\tau;\upsilon}(\F_q)|=\theta_{(\tau;\upsilon)}(q)$ for every prime power $q$.
\end{prop}
\begin{proof}
As mentioned above, the uniqueness in \eqref{eqn:innerprod} shows that
$\lambda_{(\tau;\upsilon)}$ is an element of $\Q(t)$, and hence of
$\Q(t)\cap R=\Z[t,t^{-1}]$; moreover, uniqueness implies that it does not
depend on 
the prime power $q$ used to define it. 
In addition, since
$\lambda_{(\tau;\upsilon)}$ is $\Z$-valued, it must actually lie in
$\Z[t]$.
Proposition \ref{prop:reps}
shows that each side of \eqref{eqn:target} is unchanged under 
$t\mapsto -t$, so uniqueness also shows that $\lambda_{(\tau;\upsilon)}$
is unchanged under $t\mapsto -t$, which means that 
$\lambda_{(\tau;\upsilon)}\in\Z[t^2]$. This gives the statement.
\end{proof}
\noindent
In the case of an ordinary nilpotent orbit $\cO_\lambda$, this result is well
known: we have $|\cO_\lambda(\F_q)|=\frac{a_{(1^n)}(q)}{a_\lambda(q)}$,
where $a_\lambda(t)\in\Z[t]$ is defined by \cite[II.(1.6)]{macdonald},
and it is easy to see that $\frac{a_{(1^n)}(t)}{a_\lambda(t)}\in\Z[t]$.

Along similar lines, we can
relate our intersection cohomology
to the usual Kostka polynomials via certain generalizations of
Hall polynomials.
\begin{prop}
Let $(\tau;\upsilon),(\rho;\sigma)\in\cQ_n$.
\begin{enumerate}
\item
There is a polynomial
$g_{\rho;\sigma}^{\tau;\upsilon}(t)\in\Z[t]$ such that for any 
prime power $q$ and $(v,x)\in\cO_{\tau;\upsilon}(\F_q)$,
$g_{\rho;\sigma}^{\tau;\upsilon}(q)$ counts the $\F_q$-points of the variety
\[
\{W\subset V\,|\,v\in W,\text{ Jordan type of }x|_W\text{ is }\rho,
\text{ Jordan type of }x|_{V/W}\text{ is }\sigma\}.
\]
\item 
We have
\[
\IC^{\tau;\upsilon}_{\rho;\sigma}(t)=
\sum_{\substack{\theta\leq\rho\\\psi\leq\sigma}} t^{-n(\rho+\sigma)}
g_{\theta;\psi}^{\tau;\upsilon}(t)
\tK_{\rho\theta}(t)\tK_{\sigma\psi}(t).
\]
\end{enumerate}
\end{prop}
\begin{proof}
Note that $\tK_{\mu\mu}(t)=t^{n(\mu)}$, so the transition matrix in (2) between
$\IC$ and $g$ is
a unitriangular matrix over $\Z[t]$. So if we 
define $g_{\rho;\sigma}^{\tau;\upsilon}(q)$ by the rule in (1), 
all we need to prove is that
\begin{equation} \label{eqn:desired}
\IC^{\tau;\upsilon}_{\rho;\sigma}(q)=
\sum_{\substack{\theta\leq\rho\\\psi\leq\sigma}} q^{-n(\rho+\sigma)}
g_{\theta;\psi}^{\tau;\upsilon}(q)
\tK_{\rho\theta}(q)\tK_{\sigma\psi}(q),
\end{equation}
and the fact that
$g_{\rho;\sigma}^{\tau;\upsilon}(q)$ is an integer polynomial in $q$
will automatically follow.
Now \eqref{eqn:desired}
is the characteristic-function analogue of Proposition \ref{prop:bimodule}, so we
mimic the proof of that result. The analogues of \eqref{eqn:littlekostka},
which are special cases of \cite[Lemma 5.5]{hend:ft}, are:
\begin{equation}
\begin{split}
\Pi_{\mu;\varnothing}^{\varnothing;\theta}(q)&=\sum_{\theta\leq\rho\leq\mu}q^{-n(\mu)}
K_{\rho^\bt\mu^\bt}\tK_{\rho\theta}(q),\\
\Pi_{\varnothing;\nu}^{\varnothing;\psi}(q)&=\sum_{\psi\leq\sigma\leq\nu}q^{-n(\nu)}
K_{\sigma^\bt\nu^\bt}\tK_{\sigma\psi}(q).
\end{split}
\end{equation}
The analogue of \eqref{eqn:bigkostka} follows from \eqref{eqn:fibrepurity}, 
by classifying the
$\F_q$-points of $\pi_{\mu;\nu}^{-1}(v,x)$ according to the
Jordan types of $x|_{V_{\mu_1}}$ and $x|_{V/V_{\mu_1}}$:
\begin{equation}
\begin{split}
\Pi^{\tau;\upsilon}_{\mu;\nu}(q)
&=\sum_{(\theta;\psi)\in\cQ_n}g_{\theta;\psi}^{\tau;\upsilon}(q)
\Pi_{\mu;\varnothing}^{\varnothing;\theta}(q)
\Pi_{\varnothing;\nu}^{\varnothing;\psi}(q)\\
&=\sum_{\substack{\theta\leq\rho\leq\mu\\\psi\leq\sigma\leq\nu}}
q^{-n(\mu+\nu)}
K_{\rho^\bt\mu^\bt}K_{\sigma^\bt\nu^\bt}\,
g_{\theta;\psi}^{\tau;\upsilon}(q)\tK_{\rho\theta}(q)\tK_{\sigma\psi}(q).
\end{split}
\end{equation}
Using the unitriangularity in
Corollary \ref{cor:oddzero}(3), we deduce \eqref{eqn:desired}.
\end{proof}
\noindent
Note that $g_{\rho;\sigma}^{\varnothing;\pi}(t)$ is the usual
Hall polynomial $g_{\rho\sigma}^\pi(t)$, as in \cite[II.4]{macdonald}.
The relationship between our generalized Hall polynomials and those defined in
\cite[\S4]{fgt} is that $g_{\rho;\sigma}^{\tau;\upsilon}(q)=
\sum_{\mu+\nu=\rho}G_{(\mu,\nu)\sigma}^{(\tau,\upsilon)}$.
\section{Connections with Kato's Exotic Nilpotent Cone}
\label{sect:exotic}
In this section, we discuss the analogy between the enhanced nilpotent cone
$V\times\cN$ studied in this paper and the exotic nilpotent cone
$\fN$ studied by Kato in \cite{kato:exotic},
\cite{kato:springer}, \cite{kato:deformations}. 
(We assume that $\mathrm{char}\,\F\neq 2$.)
To make a concrete connection,
we choose the symplectic vector space $W$ to be $V\oplus V^*$, with the
skew-symmetric form
\begin{equation} 
\langle(v,f),(v',f')\rangle=f'(v)-f(v').
\end{equation}
Recall that $\fN_0$ denotes the closed subvariety of $\cN(W)$ consisting of
elements which are self-adjoint for $\langle\cdot,\cdot\rangle$, and
$\fN=W\times\fN_0$.
Let $K$ denote the
symplectic group $\mathrm{Sp}(W,\langle\cdot,\cdot\rangle)$; then
$K$ clearly acts on $\fN_0$ and $\fN$.

We let $G=\GL(V)$ act on $W$ in the natural way; the resulting representation
$G\to\GL(W)$ identifies $G$ with the subgroup $\{g\in K\,|\,gV=V,gV^*=V^*\}$ of $K$.
Similarly, the map $\End(V)\to\End(W):x\mapsto(x,x^\bt)$ identifies
$\cN$ with 
\[ \{x\in\fN_0\,|\,x(V)\subseteq V,x(V^*)\subseteq V^*\}, \] 
a $G$-stable closed subvariety of $\fN_0$.
So the exotic nilpotent cone $\fN$
is sandwiched between two enhanced nilpotent cones:
$V\times\cN$ is a $G$-stable closed subvariety of $\fN$, and $\fN$ is a
$K$-stable closed subvariety of $W\times\cN(W)$.
Kato has proved that the orbits of these three varieties match up as follows.
(Here, if $\lambda$ is a partition, $\lambda\cup\lambda$ denotes the partition
$(\lambda_1,\lambda_1,\lambda_2,\lambda_2,\cdots)$.)
\begin{thm} \label{thm:kato}
The $K$-orbits in $\fN$ are in bijection with $\cQ_n$, in such a way that the orbit
$\Oo_{\mu;\nu}$ corresponding
to $(\mu;\nu)$ contains the $G$-orbit $\cO_{\mu;\nu}$, and is contained in the 
$\GL(W)$-orbit $\cO_{\mu\cup\mu;\nu\cup\nu}$.
\end{thm}
\begin{proof}
Kato's results \cite[Theorem 1.9]{kato:exotic} and \cite[Theorem B]{kato:springer} 
are not stated in quite these terms, so let us indicate how they imply the above
statement, using Proposition \ref{prop:param} to simplify the argument.
The key claim is that $\fN=K.(V\times\cN)$, which is equivalent to
saying that for any $(v,x)\in\fN$, there are $x$-stable maximal isotropic
subspaces $W_1,W_2\subset W$ such that $v\in W_1$ and $W_1\oplus W_2=W$.
Kato proves this in \cite[Appendix A]{kato:exotic} by showing that the
$K$-orbit of $(v,x)$ contains an explicit ``normal form'', which manifestly
has this property. Hence every $K$-orbit in $\fN$ contains 
a $G$-orbit in $V\times\cN$, and is contained in a unique $\GL(W)$-orbit
in $W\times\cN(W)$.
Given the parametrization of $G$-orbits in $V\times\cN$ by
$\cQ_n$ and the parametrization of $\GL(W)$-orbits in $W\times\cN(W)$ by
$\cQ_{2n}$, the result will follow immediately once we show that the orbit
$\cO_{\mu;\nu}$, regarded as a subvariety of $W\times\cN(W)$, is contained in
$\cO_{\mu\cup\mu;\nu\cup\nu}$. Take $(v,x)\in\cO_{\mu;\nu}$, and let
$\{v_{ij}\}$ be a normal basis of $V$ for $(v,x)$. If $\{v_{ij}^*\}$ denotes the
dual basis of $V^*$, then 
\begin{equation} 
x^\bt v_{ij}^*=\begin{cases}
v_{i,j+1}^*&\text{ if }j<\mu_i+\nu_i,\\
0&\text{ if }j=\mu_i+\nu_i.
\end{cases}
\end{equation}
Hence we have a Jordan basis of type $(\mu+\nu)\cup(\mu+\nu)$
for $x$ regarded as an endomorphism of $W$,
where each Jordan block $v_{i1},v_{i2},\cdots,v_{i,\mu_i+\nu_i}$ is followed
by the corresponding dual basis elements in reverse order. Applying to this basis
the normalization procedure in Lemma \ref{lem:jordan}, we obtain a normal
basis for $(v,x)\in W\times\cN(W)$ of type $(\mu\cup\mu;\nu\cup\nu)$, and the
proof is complete. 
\end{proof}
Note that in proving \cite[Theorem B]{kato:springer}, Kato
constructs a bijection between $\cQ_n$ and a set of ``marked partitions'',
and uses the latter to parametrize his normal forms; his bijection is such that
the normal form attached to $(\mu;\nu)\in\cQ_n$ is in our orbit $\Oo_{\nu;\mu}$, 
so his bipartitions need to be switched when comparing with this paper.

In order to prove that the closure ordering on the $K$-orbits in $\fN$ is given by
the same partial order as for the enhanced nilpotent cone, we need a new interpretation
of the quantity $\mu_1+\nu_1+\cdots+\mu_k+\nu_k+\mu_{k+1}$. For any
subspace $U\subset W$, $U^\perp$ denotes the perpendicular subspace under 
$\langle\cdot,\cdot\rangle$.
\begin{lem} \label{lem:symp}
For any $k\geq 0$ and $(v,x)\in\Oo_{\mu;\nu}$, 
$2(\mu_1+\nu_1+\cdots+\mu_k+\nu_k+\mu_{k+1})$ is the
maximum possible dimension of $U/(U\cap U^\perp)$ where $U$ is an
$\F[x]$-submodule of $W$ of the form $\F[x]\{v,w_1,w_2,\cdots,w_{2k+1}\}$ for some
$w_1,\cdots,w_{2k+1}\in W$.
\end{lem}
\begin{proof}
By $K$-equivariance,
we can assume that $(v,x)\in\cO_{\mu;\nu}$. 
As in the previous proof,
let $\{v_{ij}\}$ be a normal basis of $V$ for $(v,x)$, and let $\{v_{ij}^*\}$
be the dual basis of $V^*$.
We can easily see that the stated dimension is attained: set
\[
\begin{split} 
U_0&=\F[x]\{v,v_{1,\mu_1+\nu_1},\cdots,v_{k,\mu_k+\nu_k},v_{11}^*,\cdots,v_{k1}^*,
v_{k+1,1}^*\}\\
&=\Span\{v_{ij}\,|\,1\leq i\leq k,\,1\leq j\leq\mu_i+\nu_i\}
\oplus\F[x]\sum_{i=k+1}^{\ell(\mu)}v_{i,\mu_i}\\
&\qquad \oplus\Span\{v_{ij}^*\,|\,1\leq i\leq k+1,\,1\leq j\leq\mu_i+\nu_i\}.
\end{split}
\]
(If $k\geq\ell(\mu)$, ignore the middle summand; and if $k\geq\ell(\mu+\nu)$,
interpret $v_{ij}$ and $v_{ij}^*$ as zero for $i>\ell(\mu+\nu)$.) We have
$U_0\cap U_0^\perp=\Span\{v_{k+1,j}^*\,|\,j>\mu_{k+1}\}$, so 
$\dim U_0/(U_0\cap U_0^\perp)=2(\mu_1+\nu_1+\cdots+\mu_k+\nu_k+\mu_{k+1})$ as required.

We now show that for any $U=\F[x]\{v,w_1,w_2,\cdots,w_{2k+1}\}$, the dimension
of $U/(U\cap U^\perp)$ has the claimed upper bound. Notice first that $\F[x]v$ is
an isotropic subspace of $W$, being contained in the maximal isotropic subspace $V$.
The form 
$\langle\cdot,\cdot\rangle$ on $W$ induces a nondegenerate skew-symmetric form on
the subquotient $\widetilde{W}=(\F[x]v)^\perp/\F[x]v$, and $x$ induces a self-adjoint
nilpotent endomorphism of $\widetilde{W}$. Since $V/\F[x]v$ is a maximal
isotropic subspace of $\widetilde{W}$ with an $x$-stable complementary 
isotropic subspace, 
Lemma \ref{lem:jordan-dist} implies that the Jordan type of
$x$ on $\widetilde{W}$ is
\[ (\nu_1+\mu_2,\nu_1+\mu_2,\nu_2+\mu_3,\nu_2+\mu_3,\cdots). \]
Let $\widetilde{U}=(U\cap(\F[x]v)^\perp)/\F[x]v$, which is an $x$-stable subspace
of $\widetilde{W}$. Since $U/\F[x]v$ is generated as an $\F[x]$-module by the images
of $w_1,\cdots,w_{2k+1}$, $x$ has at most $2k+1$ Jordan blocks on 
$U/\F[x]v$, and hence also has at most $2k+1$ Jordan blocks on $\widetilde{U}$ and on
$\widehat{U}=\widetilde{U}/(\widetilde{U}\cap\widetilde{U}^\perp)$.
But the induced skew-symmetric form on 
$\widehat{U}$ is again nondegenerate,
so the Jordan type of $x$ on $\widehat{U}$ must be of the form $\pi\cup\pi$,
where $\ell(\pi)\leq k$. Moreover, since
the $\F[x]$-module
$\widehat{U}$ is a subquotient of $\widetilde{W}$, the Young diagram of $\pi\cup\pi$
must be contained in that of $(\nu_1+\mu_2,\nu_1+\mu_2,\cdots)$. So
\begin{equation} \label{eqn:uno}
\dim\widehat{U}=2|\pi|\leq 2(\nu_1+\mu_2+\cdots+\nu_k+\mu_{k+1}).
\end{equation} 
To relate this to $\dim U/(U\cap U^\perp)$, notice that
\[
\widetilde{U}\cap\widetilde{U}^\perp=
\frac{U\cap(\F[x]v)^\perp\cap(U^\perp+\F[x]v)}{\F[x]v}
=\frac{(U\cap U^\perp)+\F[x]v}{\F[x]v},
\]
where we have used the inclusions $\F[x]v\subseteq U\cap(\F[x]v)^\perp$ and
$U^\perp\subseteq(\F[x]v)^\perp$. So
\begin{equation} \label{eqn:due}
\dim\widehat{U}=\dim\frac{U\cap(\F[x]v)^\perp}{(U\cap U^\perp)+\F[x]v}
\geq\dim\frac{U}{U\cap U^\perp}-2\mu_1,
\end{equation}
since $\dim\F[x]v=\codim\,(\F[x]v)^\perp=\mu_1$.
Combining \eqref{eqn:uno} and \eqref{eqn:due}, we get the desired upper bound.
\end{proof}

\begin{thm} \label{thm:moby-dick}
For $(\rho;\sigma),(\mu;\nu)\in\cQ_n$, 
$\Oo_{\rho;\sigma}\subseteq\overline{\Oo_{\mu;\nu}}$ 
if and only if $(\rho;\sigma)\leq(\mu;\nu)$.
\end{thm}
\begin{proof}
It is clear that $K.\overline{\cO_{\mu;\nu}}\subseteq\overline{\Oo_{\mu;\nu}}$,
so the ``if'' direction is a consequence of the ``if'' direction in Theorem
\ref{thm:closure}. For the ``only if'' direction, 
it is clear that $\GL(W).\overline{\Oo_{\mu;\nu}}\subseteq
\overline{\cO_{\mu\cup\mu;\nu\cup\nu}}$, so we have
\[ \Oo_{\rho;\sigma}\subseteq\overline{\Oo_{\mu;\nu}}\Rightarrow
(\rho\cup\rho;\sigma\cup\sigma)\leq(\mu\cup\mu;\nu\cup\nu). \]
The latter condition does not imply $(\rho;\sigma)\leq(\mu;\nu)$,
but at least it does imply $\rho+\sigma\leq\mu+\nu$, which leaves only the
inequalities
\[ \rho_1+\sigma_1+\cdots+\rho_k+\sigma_k+\rho_{k+1}
\leq \mu_1+\nu_1+\cdots+\mu_k+\nu_k+\mu_{k+1}, \]
for all $k\geq 0$. By Lemma \ref{lem:symp}, we need to prove that for fixed $N$,
the condition
\begin{equation} \label{eqn:otherdimcond}
\begin{split}
\dim U/(U\cap U^\perp)\leq N&\text{ where }U=\F[x]\{v,w_1,\cdots,w_{2k+1}\},\\
&\text{for any $w_1,\cdots,w_{2k+1}\in W$}
\end{split}
\end{equation}
is a closed condition on $(v,x)$ (\textit{i.e.}, it determines a closed subvariety
of $\fN$).
But $\F[x]\{v,w_1,\cdots,w_{2k+1}\}$ is guaranteed to be spanned by the
$(2k+2)n$ vectors
\[ v,xv,x^2v,\cdots,x^{n-1}v,w_1,xw_1,\cdots,x^{n-1}w_1,\cdots,w_{2k+1},xw_{2k+1},
\cdots,x^{n-1}w_{2k+1}, \]
and the dimension involved in \eqref{eqn:otherdimcond} is the rank of the
$(2k+2)n\times (2k+2)n$ matrix formed by using these vectors as the left
and right inputs of $\langle\cdot,\cdot\rangle$. So as in the proof of
Theorem \ref{thm:closure}, the condition \eqref{eqn:otherdimcond} amounts to
a collection of polynomial equations in the coordinates of $v$, $x$, and
$w_1,\cdots,w_{2k+1}$, and we are done.
\end{proof}

The closures $\overline{\Oo_{\varnothing;\lambda}}$ are known to have
the same intersection cohomology as the ordinary nilpotent orbit closures
$\overline{\cO_{\varnothing;\lambda}}$, but with all degrees doubled; a proof with
a gap was given in
\cite{grojnowski}, and the gap was filled in \cite{hend:ft}. On the evidence of
direct calculations for $n\leq 3$, we conjecture that the same holds for
all $(\mu;\nu)$. In view of Theorem \ref{thm:main}, this is equivalent to
the following.
\begin{conj} \label{conj:doubling}
\begin{enumerate}
\item For $(\mu;\nu)\in\cQ_n$,
$\cH^i\IC(\overline{\Oo_{\mu;\nu}},\Qlb)=0$ for $4\nmid i$.
\item
For $(\rho;\sigma),(\mu;\nu)\in\cQ_n$ and $(v,x)\in\Oo_{\rho;\sigma}$,
\[ 
t^{b(\mu;\nu)}\sum_i \dim \cH^{4i}_{(v,x)}\IC(\overline{\Oo_{\mu;\nu}},\Qlb)\,t^{2i}
=\tK_{(\mu;\nu),(\rho;\sigma)}(t).
\]
\end{enumerate}
\end{conj}

We now sketch a possible argument to show that this conjecture is
equivalent to a recent conjecture of Shoji, stated below.

\textbf{Step 1.}
It follows from the properties of the usual Springer correspondence in type
A that for $\lambda\in\cP_n$, $x\in\cN$, and $i\geq 0$,
\begin{equation}
\dim\cH^{i}_x\IC(\overline{\cO_\lambda},\Qlb)
=\langle H^{i+2n(\lambda)}(\cB_x,\Qlb), \chi^\lambda\rangle_{S_n},
\end{equation} 
where $\cB_x$ denotes the Springer fibre ($\psi_{\varnothing;(n)}^{-1}(x)$
in the notation used before), on whose cohomology
$S_n$ acts via the Springer representation. Since $\cB_x$ has an affine paving,
both sides vanish if $i$ is odd.
Analogously, one may hope to deduce from Kato's exotic Springer correspondence in type
C that for $(\mu;\nu)\in\cQ_n$, $(v,x)\in\fN$, and $i\geq 0$,
\begin{equation} \label{eqn:step1}
\dim \cH^{i}_{(v,x)}\IC(\overline{\Oo_{\mu;\nu}},\Qlb)
=\langle H^{i+2b(\mu;\nu)}(\cC_{(v,x)},\Qlb), \chi^{\mu;\nu}\rangle_{W_n},
\end{equation}
where $\cC_{(v,x)}$ is Kato's analogue of the Springer fibre. It is also
expected that $\cC_{(v,x)}$ has an affine paving, so that both sides
would vanish if $i$ is odd.

\textbf{Step 2.}
The Springer representations in type A are isomorphic to representations 
defined purely algebraically. Explicitly, consider the graded $S_n$-module
$R_\bullet^\pi=\Qlb[x_1,\cdots,x_n]/I^\pi$,
where $I^\pi$ is the ideal of all polynomials $p(x_1,\cdots,x_n)$ such that
$p(\frac{\partial}{\partial x_1},\cdots,\frac{\partial}{\partial x_n})$
annihilates the Specht module $V^\pi$, realized in the usual way as
a subspace of the homogeneous component $\Qlb[x_1,\cdots,x_n]_{n(\pi)}$. 
It follows from~\cite{dp} that for $x\in\cO_\pi$ and for each $i$,
there is an isomorphism of $S_n$-modules
\begin{equation}
H^{2i}(\cB_x, \Qlb) \cong R_i^\pi.
\end{equation}
Analogously, one may expect that for $(v,x)\in\Oo_{\rho;\sigma}$ and for each $i$,
there is an isomorphism of $W_n$-modules
\begin{equation}\label{eqn:step2}
H^{2i}(\cC_{(v,x)},\Qlb) \cong R_i^{\rho;\sigma},
\end{equation}
where $R_\bullet^{\rho;\sigma}$ is associated in the
same way
to the Specht module $V^{\rho;\sigma}$,
realized via Macdonald--Lusztig--Spaltenstein induction (see \cite{gp}) as 
a subspace of the homogeneous component $\Qlb[x_1,\cdots,x_n]_{b(\rho;\sigma)}$.

Assuming Steps 1 and 2 can be carried out, we see
(using Corollary~\ref{cor:shoji}(1)) that
Conjecture~\ref{conj:doubling} is equivalent to the following statement:

\begin{conj}[{Shoji~\cite[3.13]{shoji:limit}}]\label{conj:shoji}
For $(\rho;\sigma),(\mu;\nu)\in\cQ_n$,
\begin{equation*} \label{eqn:step3}
\sum_i \langle R_i^{\rho;\sigma}, \chi^{\mu;\nu}\rangle_{W_n}\, t^i
=\tK_{(\mu;\nu),(\rho;\sigma)}(t).
\end{equation*}
\end{conj}

It is not clear which of these putatively equivalent conjectures would
be easier to prove.  Garsia and Procesi (see
\cite[(I.8)]{garsiaprocesi})
have given a purely algebraic/combinatorial proof that
\begin{equation}
\sum_i \langle R_i^\pi, \chi^\lambda\rangle_{S_n}\, t^i=\tK_{\lambda\pi}(t).
\end{equation}
Possibly their arguments can be adapted to prove
Conjecture~\ref{conj:shoji}.
Alternatively, one might tackle Conjecture~\ref{conj:doubling} by
imitating Lusztig's study of intersection cohomology in~\cite{lcs5}.  In
{\it loc.~cit.}, Lusztig defines a certain inner product on intersection
cohomology complexes, and then computes this inner product with the aid of
orthogonality relations for character sheaves.  An analogous inner product
for $\fN$ is defined by
\begin{multline*}
\langle \IC(\overline{\Oo_{\mu;\nu}}, \Qlb),
\IC(\overline{\Oo_{\mu';\nu'}}, \Qlb) \rangle_q\\
 = \sum_{\substack{i,j \in \Z \\ z \in \fN(\F_q)}} (-1)^{i+j} \tr(F |
\cH^i_z\IC(\overline{\Oo_{\mu;\nu}},\Qlb)) \tr(F |
\cH^i_z\IC(\overline{\Oo_{\mu';\nu'}},\Qlb)),
\end{multline*}
and the desired formula is
\begin{equation}\label{eqn:exinnerprod}
\langle \IC(\overline{\Oo_{\mu;\nu}}, \Qlb),
\IC(\overline{\Oo_{\mu';\nu'}},\Qlb) \rangle_q =
q^{-(b(\mu;\nu)+b(\mu';\nu'))}
\omega_{(\mu;\nu),(\mu';\nu')}(q).
\end{equation}
By the same uniqueness argument as in the proof of Theorem~\ref{thm:ls},
equation~\eqref{eqn:exinnerprod} would imply Conjecture~\ref{conj:doubling},
and one could also deduce as a by-product that
$|\Oo_{\mu;\nu}(\F_q)| = \theta_{\mu;\nu}(q^2)=|\cO_{\mu;\nu}(\F_{q^2})|$.

\end{document}